\documentclass{amsart}
\usepackage{amssymb,amsthm, amsmath}
\usepackage{hyperref}

\usepackage{bbold}

\usepackage[T1]{fontenc}
\usepackage{cancel}
\usepackage{graphicx}
\usepackage{enumerate}
\usepackage{thm-restate}

\newtheorem{theorem}{Theorem}[section]
\newtheorem{assumption}[theorem]{Assumption}
\newtheorem{lemma}[theorem]{Lemma}
\newtheorem{corollary}[theorem]{Corollary}
\newtheorem{proposition}[theorem]{Proposition}

\newtheorem{claim}{Claim}[theorem]

\theoremstyle{definition}
\newtheorem{remark}[theorem]{Remark}
\newtheorem{definition}[theorem]{Definition}
\newtheorem{example}[theorem]{Example}

\newtheorem{defobs}[theorem]{Definition/Observation}

\DeclareMathOperator{\Jn}{Join}
\DeclareMathOperator{\Cay}{Cay}
\DeclareMathOperator{\Carr}{Carr}
\DeclareMathOperator{\CAT}{CAT}
\newcommand{\Stab}{\operatorname{Stab}}

\newcommand{\bN}{\mathbb{N}}
\newcommand{\mc}[1]{\mathcal{#1}}

\newcommand{\ThetaD}{\Theta_{\mc{D}}}
\newcommand{\diam}{\operatorname{diam}}
\newcommand{\minus}{\smallsetminus}

\newcommand{\bndry}{\partial}
\newcommand{\Ld}[1]{\Lambda(#1)}
\newcommand{\from}{\colon\thinspace}
\newcommand{\acts}{\curvearrowright}

\usepackage{ifthen}

\newcommand{\showcomments}{yes}

\newsavebox{\commentbox}
{\ifthenelse{\equal{\showcomments}{yes}}%
{\footnotemark
    \begin{lrbox}{\commentbox}
    \begin{minipage}[t]{1.25in}\raggedright\sffamily\upshape\tiny
    \footnotemark[\arabic{footnote}]}
{\begin{lrbox}{\commentbox}}}%
{\ifthenelse{\equal{\showcomments}{yes}}%
{\end{minipage}\end{lrbox}\marginpar{\usebox{\commentbox}}}
{\end{lrbox}}}
\usepackage{todonotes}

\newcommand{\set}[2]{\{\,{#1} \mid {#2} \,\}}

\newcommand{\D}{\mc{D}}

\begin{document}

\author[M. Haulmark]{Matthew Haulmark}
\address{University of Texas Rio Grande Valley}
\email{matthew.haulmark01@utrgv.edu}

\title{From cut sets to cube complexes}

\author[J. F. Manning]{Jason Fox Manning}
\address{Cornell University}
\email{jfmanning@cornell.edu}

\begin{abstract}
In this paper we obtain an action on a cube complex from an action on a path-connected topological space with a system of \emph{divisions}. In the settings of hyperbolic groups or relatively hyperbolic groups with no peripheral splittings, our result provides an alternate route to Sageev's construction of a cube complex action from a collection of (relatively) quasiconvex subgroups of a (relatively) hyperbolic group. 
\end{abstract}

\thanks{J.M. was partly supported by the Simons Foundation through Simons Collaboration Grants \#524176 and \#942496.}

\maketitle

\pagestyle{myheadings}

\section{Introduction}

In this note we obtain an action on a cube complex from an action on a path-connected topological space satisfying certain fairly weak conditions (see \ref{as:main}). Our construction is applicable, in particular, to boundaries of hyperbolic groups and relatively hyperbolic groups which have no peripheral splitting.
In this case we give some additional information and a new point of view on some well-known results; see for example Theorems~\ref{thm: intersection of halfspaces} and~\ref{thm: JSJ Tree} below.

Because our assumptions are so weak, though, we anticipate being able to apply this framework outside the (relatively) hyperbolic world. 
Our construction, like the \emph{minimal cubing} construction of \cite{NSSS}, can be expected to give better control on the dimension and combinatorics of the dual cube complex as compared with the original formulation by Sageev~\cite{Sageev97}.
In order to state the first main result (Theorem~\ref{thm:sageev}) we must first give some definitions.

\begin{definition}
\label{def: division}
  Let $M$ be a topological space.  A \emph{division} of $M$ is a pair $D = (C_D,\{M_D^+,M_D^-\})$ so that
  \begin{enumerate}
  \item $C_D\subset M$ is closed and nowhere dense.
  \item $\{M_D^+, M_D^-\}$ is a decomposition of $M\minus C_D$ into nonempty clopen sets.
  \item For each $\epsilon\in\{+,-\}$, $\overline{M_D^{\epsilon}} = M_D^{\epsilon}\cup C_D$.
  \end{enumerate}
  Given a $G$--action on $M$, we say that $D$ is \emph{$G$--full} if its stabilizer is finite index in the $G$--stabilizer of $C_D$.
\end{definition}
If there exists a division $D = (C_D,\{M_D^+,M_D^-\})$, we say $C_D$ is a \emph{cut set} in $M$.

\begin{definition}
  A collection $\mc{C}$ of subsets of a space $M$ is \emph{point-convergent} if for every countable sequence $\{C_i\}$ of distinct elements of $\mc{C}$,
  \[ \#\left(\lim_{i\to\infty}C_i \right) \le 1 .\]
  (The limit of the collection is the set of points in $M$ which are limits of sequences $\{x_i\in C_i\}_{i\in\bN}$.)

  A collection of divisions $\mc{D}$ is \emph{point-convergent} if $\{D\mid C_D = C\}$ is finite for any $C\subset M$ and the collection $\{C_D\mid D\in \mc{D}\}$ is point-convergent.
\end{definition}
\begin{remark}\label{rem: ptconv iff null}

  If $M$ is a compact metric space, $\mc{C}$ is point-convergent if and only if it is \emph{null} in the following sense.
\end{remark}
\begin{definition}\label{def: null family} A collection of subsets $\mathcal{S}$ of a metric space is called a \emph{null family} if for every $\epsilon>0$ there are only finitely many $S\in\mathcal{S}$ with $\diam(S)>\epsilon$.\end{definition}

\begin{definition}
 Let $Y$ be a path connected topological space.  A point $p\in Y$ is a \emph{path cut point} if there exist $a,b\in Y\minus\{p\}$ so that every path from $a$ to $b$ passes through $p$.
\end{definition}
\begin{remark}\label{rem:pathcut}
  If $Y$ is locally path connected and metrizable, then every path cut point is actually a cut point; see Lemma \ref{lem:pathcut}.  For an example of a path cut point which is not a cut point, consider a point in the middle of the vertical line segment of the topologist's sine curve.
\end{remark}

The following seem to be the minimal assumptions for our approach.
\begin{assumption}\label{as:main}
  Let $M$ be a non-degenerate, path connected, $T_1$, Baire space with no path cut point, and suppose that a countable group $G$ acts by homeomorphisms on $M$.  Let $\mc{D}$ be a $G$--finite, point-convergent collection of $G$--full divisions of $M$.
\end{assumption} 

\begin{restatable}{thmA}{mainconstruction}\label{thm:sageev}
  Under Assumption \ref{as:main}, there is a $G$--action on a $\CAT(0)$ cube complex $X=X(\mc{D})$ so that the set of hyperplane stabilizers is equal to the set of division stabilizers.
\end{restatable}

Even though we do not assume $M$ is the boundary of a hyperbolic space, it turns out to be natural to work with the space of triples.  In Section~\ref{sec: cubes} we explain how a collection of divisions gives a collection of decompositions of a subset of the triple space. We then use these decomposition to put a wallspace structure on this subset.  The complex in Theorem~\ref{thm:sageev} is the cube complex dual to this wallspace.

\begin{remark}
   Our approach to building a cube complex from a collection of divisions generalizes a construction of Cashen--Macura in~\cite{CashenMacura11}.  The main difference is that Cashen--Macura focus on cut sets of valence $2$.  Our cut sets can have arbitrary valence, which means in general we must specify \emph{divisions}, not just cut sets, in order to obtain something canonical.  An exception is when the cut sets have finite valence -- this is discussed in detail in Section~\ref{sec: canonical division}. 
\end{remark}

Theorem~\ref{thm:sageev} applies in particular when $G$ has some relatively hyperbolic structure with connected Bowditch boundary and there is some collection of relatively quasiconvex subgroups whose limit sets divide that boundary.  In fact this is  the \emph{only} way Assumption \ref{as:main} can apply when $M$ is a Bowditch boundary and the sets $C_D$ have no isolated parabolic point, as demonstrated by Proposition \ref{prop:qc iff null}.
(Finite cut-sets consisting of isolated parabolic points can also be of interest;  see \cite{CashenMacura11,HPW23,kushlam}.)

In Section~\ref{sec: Limit Sets} we study the limit sets of vertex stabilizers when $M$ is compact, metric, and $G$ acts on $M$ as a convergence group.  
When $G$ is a hyperbolic group and $M = \partial G$ is its boundary at infinity, we obtain a precise description.  Given $D\in\mc{D}$, let $h_v(D)$ be either $M_D^+$ or $M_D^-$ determined by halfspace chosen by $v(D)$.

\begin{restatable}{thmA}{intersectionofhalfspaces}    
\label{thm: intersection of halfspaces}
    Suppose $G$ is hyperbolic, $M = \partial G$, and $v$ is a vertex of $X(\mc{D})$.   
    Then 
    \[ \Ld{G_v} = \bigcap_{D\in\mc{D}} h_v(D)\cup C_D.\]
\end{restatable}
Theorem~\ref{thm: intersection of halfspaces} suggests the following strategy to show a hyperbolic group is cubulated:  First find some collection of divisions so that the intersections $\bigcap_{D\in\mc{D}} h_v(D)\cup C_D$ can be controlled in such a way to ensure that any hyperbolic group with boundary equal to such an intersection is guaranteed to be cubulated (for example if the intersection is always totally disconnected, or $S^1$).  Second, apply~\cite[Theorem D]{GM23} to conclude the ambient group $G$ is cubulated.  
Since the inclusion $\Ld{G_v} \subseteq \bigcap_{D\in\mc{D}} h_v(D)\cup C_D$ is true in much greater generality  (Lemma~\ref{lem:containrevised}) the strategy might generalize to non-hyperbolic convergence groups.  We leave the application of this strategy to future work.

In general, the cube complex $X(\mc{D})$ depends on the system of divisions, and not just the collection of cut sets.  However, when all the cut sets have finite \emph{valence}, there is a natural choice of divisions.  In Section~\ref{sec: canonical division} we produce a canonical collection of divisions $\mc{D}(\mc{C})$ from a collection $\mc{C}$ of cut sets such that the number of components of $M\setminus C$ (the \emph{valence}) is finite for every $C\in\mc{C}$.  Lemma~\ref{lem: satisfies assumptions} says that if the family $\mc{C}$ is $G$--finite and point-convergent, then these canonical divisions satisfy our main assumption, and therefore give rise to a cube complex on which $G$ acts.

Adding some additional assumptions, we gain some new perspective on  results of Bowditch and Dasgupta-Hruska, as we now explain.
Assuming the elements of $\mc{C}$ do not mutually separate (see Definition~\ref{def:mutually separate}), Proposition~\ref{prop: mns is tree} shows $X\big(\mc{D}(\mc{C})\big)$ is a tree. When furthermore $(G,\mc{P})$ is relatively hyperbolic and $M$ is the Bowditch boundary, we subdivide certain edges of $X\big(\mc{D}(\mc{C})\big)$ to obtain a tree $\mc{T}(\mc{C})$. After obtaining a new peripheral structure $(G,\widehat{\mc{P}})$ from $(G,\mc{P})$ by including representatives of conjugacy classes of stabilizers of cut sets in $\mc{C}$, we establish the following.  

\begin{restatable}{thmA}{Bowditchtree}
\label{thm: JSJ Tree}
 The tree $\mc{T}(\mc{C})$ is equivariantly isomorphic to the cut point tree of the Bowditch boundary $\partial(G,\widehat{\mathcal{P}})$.   
\end{restatable} 
In the particular case that $\mc{C}$ is the collection of inseparable cut pairs in $\partial(G,\mc{P})$, we show in Corollary~\ref{cor:homeo to jsj tree} that the tree $\mc{T}(\mc{C})$ is equivariantly isomorphic to a certain canonical JSJ tree of cylinders defined by Guirardel and Levitt~\cite{GL11}. In the hyperbolic setting, this is the same as Bowditch's JSJ tree \cite{Bow98JSJ}.

\subsection{Relation to the work of Bergeron-Wise}
In~\cite{BW12}, Bergeron and Wise prove some extremely useful results about cubulating hyperbolic and relatively hyperbolic groups.  Our results do not directly generalize theirs, but their perspective was one of the inspirations for our work here, so perhaps it is useful to rephrase their results in our language.
Before doing so we make a couple of remarks.
First, these translations use Proposition~\ref{prop:qc iff null} to connect point-convergence to being the limit set of a quasiconvex subgroup.  Second, our cube complexes may not be precisely the ones defined in~\cite{BW12} as those more directly use the version of the Sageev construction based on the Cayley graph.  However their arguments will apply to any cube complex with the same (or commensurable) hyperplane stabilizers.
\begin{theorem}\cite[Theorem 1.4]{BW12}.
  Let $G$ be a hyperbolic group, and suppose there is a family $\mc{D}$ of divisions of $\partial G$ satisfying the following conditions
  \begin{enumerate}
  \item For each $D\in \mc{D}$ the subfamily $\{gD\mid g\in G\}$ is point convergent.
  \item For each pair of distinct points $x,y\in \partial G$, there is some $(C,\{M^\pm\})\in \mc{D}$ so that $x\in M^+$ and $y\in M^-$.
  \end{enumerate}
  Then there is a $G$--finite subfamily $\mc{F}\subseteq\mc{D}$ so that $G$ acts properly and cocompactly on $X(\mc{F})$.
\end{theorem}
\begin{theorem}\cite[Theorem 5.1]{BW12}
  Let $(G,\mc{P})$ be a relatively hyperbolic group pair, and let $\mc{D}$ be a family of divisions of $\partial(G,\mc{P})$.  For $P\in \mc{P}$, let $\mc{D}(P)$ be the subfamily consisting of divisions whose cut set contains the parabolic fixed point associated to $P$.  Assume the following.
  \begin{enumerate}
  \item For each $D\in \mc{D}$ the subfamily $\{gD\mid g\in G\}$ is point convergent.

  \item For each $P\in \mc{P}$,  there is a $P$--finite subfamily $\mc{F}(P)$ of $\mc{D}(P)$ so that
  $P$ acts properly on $X(\mc{F}(P))$.
  
  \item For each pair of distinct points $x,y\in \partial(G,\mc{P})$, there is some $(C,\{M^\pm\})\in \mc{D}$ so that $x\in M^+$ and $y\in M^-$.

  \end{enumerate}
  Then there is a $G$--finite subfamily $\mc{F}\subseteq\mc{D}$ so that $G$ acts properly (but not necessarily cocompactly) on $X(\mc{F})$.
\end{theorem}
Both these theorems give sufficient conditions to find a $G$--finite system of divisions with respect to which the action on the dual cube complex is proper.  In the current paper we \emph{begin} with a $G$--finite family and investigate the properties of the dual cube complex.  In some cases we can guarantee a proper action (see for example Corollary~\ref{cor:proper}).

\subsection{Acknowledgments}
We thank Ian Agol for pointing out that the Sageev construction should be describable in this way for hyperbolic groups.  We thank Annette Karrer for explaining the useful idea of \emph{point-convergence}.  Finally, we thank Nicholas Touikan and an anonymous referee for pointing out some important references we had omitted in an earlier version.

\section{Preliminaries}

\subsection{The Cube Complex Dual to a Space with Walls}\label{app:wallstocubes}

In this section we briefly recall the notion of a \emph{space with walls} and its dual cube complex \cite{CN05,Nica04}.  We refer the reader to \cite{SageevPCMI, HW14} for more details and for slightly more general situations.

\begin{definition}
    Let $Y$ be a set and let $\mc{H}$ be a subset of $2^Y\setminus\{\emptyset\}$ which is closed under the operation
    \[A \mapsto A^* = Y\smallsetminus A.\]
    If the set $\mathcal{H}$ satisfies the finiteness condition
    \begin{equation}\label{eq:finiteness} \tag{F} \#\set{A\in \mc{H}}{x\in A,y\in A^*}<\infty,\quad\forall x,y\in Y,
    \end{equation}
    Then $(Y,\mc{H})$ is said to be a \emph{space with walls}.  The pairs $\{A,A^*\}$ are the \emph{walls}, and the elements of $\mc{H}$ are called \emph{halfspaces}.
\end{definition}

\begin{definition}
    Let $(Y,\mc{H})$ be a space with walls. Two walls $\{A, A^*\}$ and $\{B,B^*\}$ are \emph{transverse} if $A\cap B$, $A^*\cap B$, $A\cap B^*$, and $A^*\cap B^*$ are all nonempty.
\end{definition}

We now describe how to construct a cube complex $X$ from a space with walls.  For this construction we must first define ultrafilters. 

\begin{definition}\label{def:ultrafilter}
    An \emph{ultrafilter} on a space with walls $(Y,\mc{H})$
    is a subset $\alpha$ of $\mc{H}$ satisfying:
        \begin{enumerate}
            \item\label{itm:complete} For every $A\in\mc{H}$ exactly one of $\{A,A^*\}$ is in $\alpha$.
            \item If $A\in\alpha$ and $A\subseteq B$, then $B\in\alpha$.
        \end{enumerate}
\end{definition}

\begin{definition}[Principal ultrafilters]\label{def:principal1}
 Let $(Y,\mc{H})$ be a space with walls and let $y\in{Y}$. Then $\alpha_{y}=\set{A\in\mc{H}}{y\in A}$ is an ultrafilter.   
\end{definition}
\begin{remark}
    Let $(Y,\mc{H})$ be a space with walls.  Denote the set of walls $\{A,A^*\}$ by $\mc{W}$.  Because of condition~\eqref{itm:complete} in Definition~\ref{def:ultrafilter} an ultrafilter determines (and is determined by) a function $f_\alpha$.  This function
    \[ f_\alpha\from \mc{W}\to \mc{H}\]
    is determined by $ W\cap \alpha = \{f_\alpha(W)\}$.  We frequently abuse notation by conflating $\alpha$ with $f_\alpha$.
\end{remark}
    
\begin{definition}
\label{def: dcc}
 An ultrafilter satisfies the \emph{Descending Chain Condition (DCC)} if every descending chain of elements terminates.    
\end{definition}
Note that the finiteness condition~\eqref{eq:finiteness} implies that principal ultrafilters are DCC.

If $\alpha$ is a DCC ultrafilter on a space with walls $(Y,\mc{H})$ and $A\in\alpha$ we set 
\[(\alpha;A)=(\alpha\setminus\{A\})\cup \{A^*\}\]

\begin{definition}
    Let $(Y,\mc{H})$ be a space with walls.  We define a cube complex $X=X(Y,\mc{H})$ by first defining a cube complex $Z$ and then taking a certain connected component:
    \begin{itemize}
        \item Let $Z^{(0)}$ be the set of DCC ultrafilters on $\mc{H}$.
        \item Two ultrafilters $\alpha$ and $\alpha'$ are connected by an edge if and only if $\alpha'=(\alpha; A)$ 
        for $A$ minimal in $\alpha$ with respect to inclusion.
        \item For $n>1$, $n$-cubes are attached whenever the $(n-1)$-skeleton of an $n$-cube appears.
    \end{itemize}
    The finiteness condition~\eqref{eq:finiteness} implies that all principal ultrafilters lie in the same connected component of $Z$ \cite[Section 3]{HW14}.
    We let $X$ be the connected component containing the principal ultrafilters.
\end{definition}
This cube complex $X$ is $\CAT(0)$; see \cite{CN05,Nica04} (see 
  \cite[Theorem 3.7]{HW14} for a slightly more general statement). 
Let $(Y,\mc{H})$ be a space with walls, and suppose $G$ acts on $Y$ preserving the set of half-spaces $\mc{H}$.  If $X$ is the dual cube complex, there is an induced action of $G$ on $X^{(0)}$ given by:

\[g\alpha = \set{gA}{A\in\alpha}\]
From the functional point of view, we have, for each wall $W$,
\[ f_{g\alpha}(W) = g f_\alpha(g^{-1}W).\]
This action preserves adjacency of vertices, so we obtain an action on the cube complex $X$.

We also point out the following consequence of the construction of $X$ (cf. \cite[Proposition 3.14]{HW14}):

\begin{proposition}\label{prop: dimension cube complex}
The dimension of the cube complex $X$ is equal to the maximum number of pairwise transverse walls.  
\end{proposition}

\begin{definition}
    The $i$th \emph{midcube} of the cube $[0,1]^n$ is the subset of points whose $i$th coordinate is equal to $\frac{1}{2}$.  
    A \emph{hyperplane} of a $\CAT(0)$ cube complex is a minimal nonempty subset whose intersection with each cube is either empty or equal to a midcube.  
\end{definition}
Hyperplanes in a $\CAT(0)$ cube complex are in one-to-one correspondence with equivalence classes of edges, where the equivalence relation is generated by $e_1\sim e_2$ whenever $e_1$ and $e_2$ are opposite sides of the same square.  

\begin{remark}\label{rem:minimality}
    Edges dual to a vertex in the cube complex coming from a wallspace are in one-to-one correspondence with the \emph{minimal} halfspaces chosen by that vertex.
\end{remark}

\section{Cubes from divisions}\label{sec: cubes}

In this section we explain how a system of divisions satisfying our main Assumption~\ref{as:main} gives rise to a wallspace with a $G$--action and hence (using Section~\ref{app:wallstocubes}) a $\CAT(0)$ cube complex with a $G$--action.

\begin{definition}
\label{def: walls on triples}
  An \emph{off-the-wall point} for $\mc{D}$ is any point not in a cut set $C_D$ for $D\in \mc{D}$.
  We write $M_\mc{D}$ for the set of off-the-wall points.  In other words,
  \[ M_\mc{D} = M \minus \bigcup_{D\in \mc{D}} C_D. \]
  Since $M$ is a Baire space and $\bigcup_{D\in \mc{D}} C_D$ is a countable union of nowhere dense subsets,
  $M_\mc{D}$ is dense in $M$.
For any space $X$, we write $\Theta(X)$ for the set of unordered triples of distinct points in $X$.  

  We define a set $\ThetaD\subset \Theta(M)$ by $\ThetaD = \Theta(M_\mc{D})$.  This $\ThetaD$ will be given the structure of a space with walls; a cube complex is then determined via the construction recalled in Section~\ref{app:wallstocubes}. 

  Given $D\in\mc{D}$, we have $M\setminus C_D = M_D^+\sqcup M_D^-$.  Define the halfspaces $W_D^\pm$ by
  \begin{align*}
    W_D^+ &= \left\{\left. \{a,b,c\}\in \ThetaD\right| \#\left(\{a,b,c\}\cap M_D^+\right) \ge 2\right\}\\
    W_D^- &= \left\{\left. \{a,b,c\}\in \ThetaD\right| \#\left(\{a,b,c\}\cap M_D^-\right) \ge 2\right\}.
  \end{align*}
  Each of $W_D^+$ and $W_D^-$ will be called a \emph{halfspace} in $\ThetaD$.
  Observe that $\ThetaD = W_D^+\sqcup W_D^-$, so we obtain a wall $W_D =\{W_D^+,W_D^-\}$.  We let $\mc{W}_{\mc{D}}$ be the collection of these walls.
\end{definition}
The relationship between $M_D^\pm$ and $W_D^\pm$ can be explained as follows.  Each triple is polled as to which side of the division it lies on, and the majority rules.  Essentially the same approach is taken in Cashen--Macura~\cite[Section 2]{CashenMacura11} in the special case of valence two cut-sets; our off-the-wall points correspond to their ``bad points''.

We next show that these decompositions make $\ThetaD$ into a space with walls.

\begin{lemma}\label{lem:wallspace}
  The collection $\mc{W}_{\mc{D}}$ gives $\ThetaD$ the structure of a space with walls.
\end{lemma}
\begin{proof}
  We need to show, that for any $\{x,y,z\}, \{a,b,c\}\in \ThetaD$, the set of divisions $D$ so that $\{x,y,z\}\in W_D^+$ and $\{a,b,c\}\in W_D^-$ is finite.  Arguing by contradiction we suppose there are infinitely many such walls and show that this forces $M$ to contain a path cut point.  

  Since the set $\mc{D}$ of divisions is assumed to be $G$--finite, we may pass to an infinite subcollection which are translates of some fixed $D\in \mc{D}$.
  After possibly relabeling points, we may suppose that, for an infinite sequence $\{g_i\}$ in $G$ we have
  \begin{equation*}
    \{x,y\}\subset g_i M_D^+\quad \mbox{and}\quad \{a,b\}\subset g_iM_D^-
  \end{equation*}
  Since the divisions are assumed to be $G$--full, $\Stab(D)$ is finite index in $\Stab(C_D)$, so after passing to a subsequence we may assume that all the $g_i C_D$ are distinct.  Let $\sigma$ be a path from $x$ to $a$.  Any path from $\{x,y\}$ to $\{a,b\}$ must meet every $g_i C_D$.  After passing to a subsequence we find a point $p\in \sigma\cap \left(\lim_{i\to\infty}g_i C_D\right)$.

  We claim that $p$ is a path cut point.  Indeed, let $\tau$ be any other path joining the set $\{x,y\}$ to the set $\{a,b\}$.  As above, we pass to a further subsequence so that $\tau\cap \left(\lim_{i\to\infty} g_i C_D\right)$ contains a point $q$.  Because of point convergence we must have $q = p$, so $p$ is contained in every path joining $\{x,y\}$ to $\{a,b\}$.  At most one of the four points $\{x,y,a,b\}$ can be equal to $p$, so there is a pair of points not equal to $p$ so that every path joining them contains $p$. Thus $p$ is a path cut point.

  This contradicts our assumption on $M$, so we have verified that $\mc{W}_{\mc{D}}$ gives $\ThetaD$ the structure of a space with walls.
\end{proof}
\begin{definition}\label{def:X(D)}
    Lemma~\ref{lem:wallspace} and the construction described in Section~\ref{app:wallstocubes} gives us a cube complex dual to the space with walls $(\ThetaD,\mc{W}_{\mc{D}})$. 
    We write $X(\mc{D})$ for this cube complex.
\end{definition}

For the convenience of the reader we restate our first main result.
\mainconstruction*

\begin{proof}
  The wallspace structure $\mc{W}_{\mc{D}}$ on $\ThetaD$ is preserved by the action of $G$ on $\ThetaD$, so $G$ also acts on the dual cube complex $X=X(\mc{D})$ from Definition~\ref{def:X(D)}.  Recall each vertex of $X$ corresponds to a certain function $\eta$ from the collection of walls $\mc{W}_{\mc{D}}$ to the collection of halfspaces so that $\eta(W)\in\{W^+,W^-\}$ for each $W$.  If $\eta$ and $\eta'$ are connected by an edge, then they differ on exactly one such $W$, so each edge determines a wall $W$.  Two edges dual to the same hyperplane determine the same wall.  Thus a hyperplane stabilizer is equal to a wall stabilizer, which is equal to a division stabilizer.

  To see each division stabilizer is also a hyperplane stabilizer, we note that every wall has at least one edge in $X(\mc{D})$ dual to it, since by Lemma~\ref{lem: Baire} there are uncountably many off-the-wall points in $M_D^+$ and $M_D^-$ for any division $D$.  In particular, there are principal ultrafilters making each possible choice on $D$. 
  
\end{proof}
\begin{remark}\label{rem:inversions}
    If a hyperplane $H$ comes from the division $D = (C_D,\{M_D^+,M_D^-\})$ we see from the above proof that the stabilizer of $D$ is the same as the stabilizer of $H$.  We will refer to $H$ as the \emph{hyperplane dual to $D$}.
    The subgroup of $\Stab(H)$ which also stabilizes each of the two half-spaces of $X(\mc{D})$ defined by $H$ is equal to $\Stab(M_D^+)=\Stab(M_D^-)$.
\end{remark}

We now give conditions under which the resulting complex is a tree.

\begin{theorem}\label{thm:tree}
  Together with Assumption~\ref{as:main}, assume the following.
  \begin{enumerate}
      \item For every $D\in \mc{D}$, the cut set $C_D$ is connected.
      
      \item\label{itm:nonsep intersection} Let $D,D'\in \mc{D}$ be distinct, and let $I = C_D\cap C_{D'}$.  Then $C_D\setminus I$ and $C_{D'}\setminus I$ are connected and $M\setminus I$ is path-connected.

  \end{enumerate}
  Then the cube complex $X(\mc{D})$ is a tree. Furthermore, if for every $D\in \mc{D}$, $\Stab(C_D) = \Stab(D)$, then $G$ acts on $X(\mc{D})$ without inversions.
\end{theorem}

\begin{proof}
  By Proposition \ref{prop: dimension cube complex} to show the cube complex is a tree we must show that no two walls are transverse.  To do this we must show that, for an arbitrary pair of walls, disjoint halfspaces can be chosen.  We will first show that the corresponding statement is true in $M$.

  Let $D$ and $D'$ be distinct divisions, giving rise to the walls $W_D$ and $W_{D'}$.  
  By~\eqref{itm:nonsep intersection}, the cut sets $C_D$ and $C_{D'}$ differ.

  Let $I = C_D\cap C_{D'}$.  By assumption~\eqref{itm:nonsep intersection}, the set $C_D\minus I$ is connected.  It is therefore contained in $M_{D'}^+$ or $M_{D'}^-$.  We choose $\epsilon\in \{\pm\}$ so that $C_D\minus I \subset M_{D'}^\epsilon$.  Similarly, we choose $\delta\in \{\pm\}$ so that $C_{D'} \minus I \subset M_{D}^\delta$.

\begin{claim}
  $M_{D'}^{-\epsilon}\cap M_D^{-\delta}=\emptyset$
\end{claim}

\begin{proof}[Proof of claim]
  Let $x\in M_{D'}^{-\epsilon}$.
  Since $C_D\minus I\subset M_{D'}^\epsilon$, any path from $C_D\minus I$ to $x$ passes through $C_{D'}$.

  On the other hand, suppose $x\in M_D^{-\delta}$.  Since $I$ doesn't path-separate $M$, there is a path $\sigma$ from $x$ to $C_D\minus I$ missing $I$.  Truncating, we may assume $\sigma$ only meets $C_D$ at its endpoint $c\ne x$.  The connected set $\sigma \minus \{c\}$ lies entirely in $M_D^{-\delta}$.  But $C_{D'}\minus I \subset M_D^\delta$, so $\sigma$ misses $C_{D'}$ and $x$ cannot be in $M_{D'}^{-\epsilon}$.  
\end{proof}
The claim implies that $W_{D'}^{-\epsilon}\cap W_D^{-\delta} = \emptyset$, so the walls $W_D$ and $W_{D'}$ do not cross, so $X(\mc{D})$ is a tree.
The statement about inversions follows from Remark~\ref{rem:inversions}.
\end{proof}

In general, the cube complex we construct may have a global fixed point and thus may not be so useful.  This happens for example if the action of $G$ on $M$ factors through a finite group.  Another way this happens is if there is some $G$--invariant way to choose halfspaces $M_D^{\pm}$ for all $D\in \mc{D}$ at once. For an explicit example, consider a hyperbolic group with boundary a Sierpinski carpet and extend the boundary action to an action on the $2$--sphere.  Now, let $\mc{D}$ be the family of divisions of the $2$--sphere coming from the boundary circles.  Orienting all these circles toward the limit set gives an ultrafilter fixed by the entire group.

Here is a criterion which rules out this kind of behavior.

\begin{proposition}\label{prop:noglobalfp}
Suppose that for some $g\in G$ and
$D\in \mc{D}$, 
\begin{enumerate}
    \item\label{itm:nest} $gM_D^+\subset M_D^+$; and
    \item\label{itm:proper} $M_D^+\setminus\overline{gM_D^+}$ is nonempty.
\end{enumerate}
If $H$ is the hyperplane corresponding to $D$, then 
$g$ acts on $X(\mc{D})$ hyperbolically, with axis crossing $H$.  In particular $H$ is an essential hyperplane.
\end{proposition}
\begin{proof}
    By Lemma~\ref{lem: Baire} and assumption~\eqref{itm:proper}, there is some triple $t = (a,b,c)$ in $\Theta_{\mc{D}}$ so that $\{a,b,c\}\subset M_D^+\setminus\overline{gM_D^+}$.  Let $\eta$ be the principal ultrafilter for $t$.  Since $t\notin g M_D^+$, we have $g^k W_D^+\notin \eta$ for $k>0$.
    The nesting~\eqref{itm:nest} implies that, for $k\ge 0$, $g^kt\in M_D^+$, so $g^{-k}W_D^+\in \eta$ and moreover that $g^kH$ is disjoint from $g^lH$ whenever $k\ne l$.  

    Multiplying by a power of $g$, we have, for any $n$
    \[ g^kW_D^+\in g^n\eta\quad\mbox{iff}\quad k\le n.\]
    (The ultrafilter $g^n\eta$ is the principal ultrafilter for $g^nt$.)  In particular $\eta$ and $g^n\eta$ differ on at least $|n|$ walls (each of which is a translate of $\{W_D^+,W_D^-\}$ by a power of $g$), and $d(g^m\eta,g^n\eta)\ge |m-n|$.  It follows that $g$ acts hyperbolically on $X(\mc{D})$ with axis crossing the hyperplane $H$.
\end{proof}

\section{Convergence groups and relative hyperbolicity}\label{sec:convergence}
In this section we recall some background on convergence groups.  In particular we give the dynamical definitions of relative hyperbolicity and relative quasiconvexity. 
Nothing in this section is really new, but for the convenience of the reader we give a proof of Proposition~\ref{prop:qc iff null}, which is a variation on a result of Haïssinsky, Paoluzzi, and Walsh.    Proposition~\ref{prop:qc iff null} describes the interaction of our main assumption with geometrically finite convergence actions.

\subsection{Convergence groups in general}\label{sec: conv groups}
In this subsection we review the basic definitions. For more on convergence groups see \cite{GehM87,Tuk94,Bow99Conv, Freden95, Freden97}.

Let $M$ be a compact metrizable space and $G$ a group acting on $M$ by homeomorphisms. 
The action of $G$ on $M$ is said to be a \emph{convergence group action} if for any infinite collection of group elements there is a sequence $\{g_i\}$ (called a \emph{convergence sequence}) and two  points $a,b$ so that 
\begin{equation}\label{eq:convergence}
    g_i x \to a,\mbox{ locally uniformly for } x\neq b.
\end{equation}  
The points $a$ and $b$ are called the \emph{attracting} and \emph{repelling} points of the sequence, respectively; they need not be distinct.
If $\{g_i\}$ is a convergence sequence, so is $\{g_i^{-1}\}$; the roles of the attracting and repelling points are reversed.  By \emph{locally uniformly} we mean, if $C$ is a compact subset of $M\setminus\{\beta\}$ and $U$ is any open neighborhood of $\alpha$, then there is an $N\in\mathbb{N}$ such that $g_{k_i} C\subset U$ for all $i>N$.
The set of possible points $a$ occurring in~\eqref{eq:convergence} is called the \emph{limit set} of $G$, and is written $\Ld{G}$. 
The complement of $\Ld{G}$ is written $\Omega(G)$, and is called the \emph{domain of discontinuity}.  Appropriately, $G$ acts properly discontinuously on $\Omega(G)$; see \cite[Theorem 2L]{Tuk94}.

Elements of convergence groups can be classified into three types: elliptic, loxodromic, and parabolic. A group element is \emph{elliptic} if it has finite order. An element $g$ of $G$ is \emph{loxodromic} if has infinite order and fixes exactly two points of $M$. In this case the sequence of positive powers of $g$ is a convergence sequence, and the two points of $\Ld{\langle g\rangle}$ are the attracting and repelling points. If $g\in G$ has infinite order and fixes a single point of $M$ then $g$ is \emph{parabolic}. An infinite subgroup $P$ of $G$ is \emph{parabolic} if it contains no loxodromic elements and stabilizes a single point $p$ of $M$.  The point $p$ is uniquely determined by $P$, and the point $p$ is called a \emph{parabolic} point. We call $p$ a \emph{bounded parabolic} point if $P$ acts properly and cocompactly on $M\setminus\{p\}$.

A point $x\in M$ is a \emph{conical limit point} point if there exists a sequence of group elements $\{g_n\}\in G$ and distinct points $a, b \in M$ such that $g_nx\rightarrow a$ and $g_ny\rightarrow b$ for every $y\in M\setminus\{x\}$. Tukia has shown (see \cite{Tukia98Conical}) that a conical limit point cannot be a parabolic point.

 A convergence group $G$ acting on $M$ is called \emph{uniform} if every point of $M$ is a conical limit point, or equivalently the action on space of distinct triples of $M$ is proper and cocompact (see \cite{Bow99Conv}). Bowditch has shown \cite{Bow98TopCha} $G$ acts as a uniform convergence group on some compact metric space if and only if $G$ is hyperbolic.

\subsection{Relative Hyperbolicity and relative quasiconvexity}

We next review the notions of relative hyperbolicity and relative quasiconvexity. Then we prove a result about the nullity of limit sets of relatively quasiconvex subgroups. 

A convergence group $G$ is called \emph{geometrically finite} if every point of $M$ is a conical limit point or a bounded parabolic point. Suppose $\mc{P}$ is a set of representatives of the conjugacy classes of maximal parabolic subgroups we say the action of $(G,\mc{P})$ on $M$ is geometrically finite. 

\begin{definition}[RH-1 in\cite{Hruska10}] \label{def:relhyp}
Suppose $(G,\mc{P})$ acts as a geometrically finite convergence group on a compact metrizable space. Then $(G,\mc{P})$ is \emph{relatively hyperbolic}. 
\end{definition}

 The space in Definition~\ref{def:relhyp} only depends up to homeomorphism on $(G,\mc{P})$ \cite{Bow12} \cite{Yaman04} and is called the \emph{Bowditch boundary}, denoted $\bndry(G,\mc{P})$. 
 We note that there are only finitely many orbits of parabolic points in $\bndry(G,\mc{P})$ (Proposition 6.15 of \cite{Bow12}).

\begin{definition}[QC-1 in \cite{Hruska10}]\label{def: dynamic rel qc}

A subgroup $H\leq  G$ is \emph{relatively quasiconvex} if the
following holds. Let $M$ be some (any) compact, metrizable space on which
$(G, \mc{P})$ acts as a geometrically finite convergence group. Then the induced
convergence action of $H$ on the limit set $\Lambda(H)\subset M$ is geometrically finite.
\end{definition}

Hruska in \cite{Hruska10} gives a number of equivalent definitions of relative hyperbolicity and relative quasiconvexity. 
Some of these make reference to a \emph{cusped space} for the pair $(G,\mc{P})$, a hyperbolic space acted on by the group $G$ so that the action on its Gromov boundary is geometrically finite.  A combinatorial construction can be found in~\cite{GM08}; see \cite[Section 3]{HH23} for a summary of a construction due to Bowditch~\cite{Bow12}.

A version of the following proposition was independently noticed by Ha\"issinsky--Paoluzzi--Walsh~\cite[Proposition 2.3]{HPW23}.
\begin{proposition}\label{prop:qc iff null}
  Let $(G,\mc{P})$ be relatively hyperbolic, let $M = \partial(G,\mc{P})$, 
  and let $C\subset M$ be a closed set with at least two points and with no isolated parabolic points.  Then $\{gC\mid g\in G\}$ is a null family if and only if $C$ is the limit set of a relatively quasiconvex subgroup.
\end{proposition}
\begin{proof}
  The ``if'' direction is well known, see for example \cite[Corollary 2.5]{GMRS98} in the absolute case.  We give a sketch of the proof.  Let $X = X(G,\mc{P},S)$ be a (combinatorial) cusped space for the pair $(G,\mc{P})$, and endow $M = \partial(G,\mc{P})$ with a visual metric $\rho$ based at the identity.  Then for any $\epsilon$ there is an $R$ so that if $\diam(g\Ld{H})>\epsilon$ there is some geodesic with endpoints in $g\Ld{H}$ which meets the $R$--ball around the identity in $X$.  Let $gH$ be one of these cosets and $\gamma$ the geodesic.  Relative quasiconvexity of $H$ implies that any point of $\gamma\cap G$ is within $\mu$ of some element of $gH$, where $\mu$ is a constant depending only on $H$ and the space $X$ (\cite[Definition QC-3]{Hruska10}).  Letting $x$ be a point on $\gamma$ so that $d_X(1,x)\le R$, the \emph{depth} of $x$ (ie distance from $G$) is at most $R$.  This implies there is a group element $y$ on $\gamma$ with $d_X(x,y) \le R+2$.  Relative quasiconvexity gives $z\in gH$ with $d_X(y,z) \le \mu$, so $gH$ meets the $(2R + \mu + 2)$--ball around the identity in $X$.  The induced $X$--metric on $G$ is proper, so there are only finitely many cosets which could meet this ball.

  In the other direction we suppose that $\{gC\mid g\in G\}$ is a null family.
    We show that $\Stab(C)$ acts as a geometrically finite convergence group on $C$.  (This is \cite[Definition QC-2]{Hruska10}, and comes from the work of Dahmani.) In other words, we show that every point of $C$ is either a conical limit point or a bounded parabolic point for the action of $\Stab(C)$.  

  Suppose first that $x\in C$ is a conical limit point for the action of $G$.
  That is to say there is a pair of distinct points $p,q$ and a sequence $\{g_i\}$ of elements of $G$ so that $g_ix\to p$ and $g_iy\to q$ for all $y\ne x$.  Since $C\minus \{x\}$ is nonempty, this implies that the diameters of the sets $g_i C$ are bounded below.  Since $\{gC\mid g\in G\}$ is a null family, this means that infinitely many must coincide.  Passing to a subsequence we may assume that $g_iC = g C$ for all $i$ and some fixed $g\in G$.  Consider the sequence $\{g_i' = g^{-1}g_i\}$.  This sequence lies in $\Stab(C)$, and still exhibits the fact that $x$ is a conical limit point.  Namely, we have $g_i'x\to g^{-1}p$ and $g_i'y\to g^{-1}q$ for all $y\ne x$, so $x$ is a conical limit point for the action of $\Stab(C)$, as desired.

  Now we must show that the remaining points are bounded parabolic.  
  Let $x$ be a point of $C$ which is not a conical limit point for the action of $G$.  Then $x$ is a bounded parabolic point for the action of $G$.  Let $P$ be the stabilizer of $x$ in $G$, and let $S<P$ be the stabilizer of $x$ in $\Stab(C)$.  
  Let $M_p = M\minus \{p\}$, and let $C_p = C\minus\{p\}$.  We first show that $S$ acts properly and cocompactly on $C_p$.  Properness follows from the properness of the action of $P$ on $M_p$.  It remains to show cocompactness.

  Let $F\subset M_p$ be a compact set so that $PF = M_p$.  The set
  \[I =  \{g\in P\mid F\cap g C_p\ne \emptyset\} \]
  is a union of cosets of $S$, say $I = \sqcup_{t\in T} tS$.  If $t\in T$, then $tC_p\cap F\ne \emptyset$, which means that $tC$ meets both $x$ and $F$.  Since this puts a lower bound on the diameter of $tC$, and $\{gC\}$ is a null family, the set $T$ is finite.  But
  \[ C_p \subset \cup_{g\in I} g^{-1}F = S \cup_{t\in T} t^{-1}F. \]
  Since $T$ is finite, $\cup_{t\in T} t^{-1}F$ is compact, and so the $S$--action on $C_p$ is cocompact.

  Since $p$ is not isolated, $S$ must be infinite, and so $p$ is a parabolic point for the action of $H$ on $C$.
\end{proof}

\subsection{Splittings of Relatively Hyperbolic Groups}\label{subsec: cut point tree}

 In this subsection we recall the notion of a peripheral splitting  \cite{Bow01_Peripheral} of a relatively hyperbolic group. This is required for the main results of Section~\ref{sec: canonical division}. A \emph{peripheral splitting} of $(G, \mc{P})$ is a finite bipartite graph of
groups representation of $G$, where $\mc{P}$ is a set of representatives of conjugacy classes of vertex groups of one color of the partition. Bowditch and Dasgupta have shown that $(G,\mc{P})$ has a peripheral splitting if and only if $M=\bndry(G,\mc{P})$ has a cut point (see Theorem~\ref{thm: peripheral splittings} below).

The below theorem due to Bowditch \cite{Bow01_Peripheral} and Dasgupta-Hruska \cite{DH24} describes the \emph{cut point tree} of $\bndry(G,\mc{P})$. The cut point tree is determined by the topology of $\bndry(G,\mc{P})$ but it encodes group theoretic information: the peripheral splittings of $G$.  Before we state the theorem we provide a definition.

\begin{definition}\label{def: cyclic element}
    A subset $A \subseteq \bndry(G,\mathcal{P})$ is a \emph{cyclic element} if $A$ consists of a single cut point or contains a non-cut point $p$ and all points $q$ that are not separated from $p$ by any cut point of $M$.
\end{definition}

\begin{theorem}[Bowditch, Dasgupta-Hruska]
\label{thm: peripheral splittings}
Let $(G,\mathcal{Q})$ be relatively hyperbolic with connected boundary $M=\bndry(G,\mathcal{Q})$.
Let $T$ be the bipartite graph with vertex set $\Pi \sqcup K$, where $\Pi$ is the set of cut points and $K$ is the set of nontrivial cyclic elements of $M$.  Two vertices $v \in \Pi$ and $w \in K$ are connected by an edge in $T$ if and only if the cut point $v$ is contained in the cyclic element $w$.

The graph $T$ is equal to the canonical JSJ tree of cylinders for splittings of $G$ over parabolic subgroups relative to $\mathcal{Q}$.
The edges of $T$ lie in finitely many $G$--orbits and the stabilizer of each edge is finitely generated.
\end{theorem}

\section{Limit Sets of Vertex Stabilizers}\label{sec: Limit Sets}
In this section we restrict to the situation in which $M$ is compact metric space and $G$ acts as a convergence group.  
We analyze the limit sets of vertex stabilizers.  (Stabilizers of higher-dimensional cubes are virtually intersections of finitely many cube stabilizers, so in principle these can also be described by the results in this section.)
First, recall: For $D\in \mc{D}$ and $v$ a vertex of $X(\mc{D})$, we write $h_v(D)$ for either $M_D^+$ or $M_D^-$ depending on which halfspace the ultrafilter $v$ chooses.  The stabilizer of a vertex $v$ is denoted $G_v$. 
Throughout this section, we make the standing assumption that $\mc{D}$ is a $G$-finite point-convergent collection of divisions, and that $X=X(\mc{D})$.

Later in the section we will specialize to the case that $M$ is the boundary of a hyperbolic or relatively hyperbolic group.  By results of Bestvina-Mess \cite{BestvinaMess91}, Swarup \cite{Swarup96}, and \cite{Bowditch99Connectedness} the boundary of a one-ended hyperbolic group is locally connected and without cut points. Similarly, Dasgupta-Hruska \cite{DH24} have shown that the Bowditch boundary of a relatively hyperbolic group is locally connected if it is connected.  
Compact connected locally connected metric spaces are locally path connected~\cite[Theorem 31.4]{Willard}.  In particular Lemma~\ref{lem:pathcut} tells us that if one of these boundaries has no cut point, it also has no path cut point.  Thus the hypotheses on $M$ in Assumption~\ref{as:main} hold for such boundaries and the cube complex $X = X(\mc{D})$ exists for any collection of divisions $\mc{D}$ satisfying the remaining assumptions.

In the generality of a convergence action satisfying our main assumptions we have the following.

\begin{lemma}\label{lem:containrevised}
Suppose, in addition to Assumption~\ref{as:main}, that $M$ is compact metric and $G$ acts on $M$ as a convergence group.  Let $v$ be a vertex of the dual cube complex.  Then
\begin{equation}\label{eq:containment}\Ld{G_v}\subseteq \bigcap_{D\in\mc{D}} h_v(D)\cup C_D.
    \end{equation}
\end{lemma}
\begin{proof}
    Our assumptions imply that, for each $D\in \mc{D}$, the closure of $h_v(D)$ is equal to $h_v(D)\cup C_D$.
    
    We may assume that $G_v$ is infinite.  If some $\overline{h_v(D)}$ is disjoint from $\Ld{G_v}$, then the proper discontinuity of the action on $\Omega(G_v)$ means there is some $g\in G_v$ so that $g\overline{h_v(D)}\cap \overline{h_v(D)}$ is empty.  But this implies that $v$ is inconsistent, a contradiction.  We conclude that $\Ld{G_v}\cap \overline{h_v(D)}$ is nonempty for each $D\in \mc{D}$.  If $\Ld{G_v}$ is a singleton, we are finished.

    Suppose then that $\Ld{G_v}$ contains at least two points.  Lemma 2Q of~\cite{Tuk94} tells us that there is a dense set in $\Ld{G_v}$ consisting of fixed points of loxodromics in $G_v$.  If we can show that each such loxodromic fixed point is contained in $\bigcap \overline{h_v(D)}$, then since that intersection is closed, we will be finished.
    
      So let $g\in G_v$ be a loxodromic, fixing $p,q\in \Ld{G_v}$, and let $D\in \mc{D}$.   
    If neither $p$ nor $q$ is contained in $\overline{h_v(D)}$, there is a power of $g$ which takes $\overline{h_v(D)}$ off of itself, violating the consistency condition again.  So suppose that $p\in \overline{h_v(D)}$, but $q\notin \overline{h_v(D)}$.  Replacing $g$ by its inverse if necessary, we may assume that $q$ is the attracting fixed point for $g$.
    Note first that $p$ cannot be contained in $C_D$, or the translates $\{g^n C_D\mid n\ge 0\}$ would give infinitely many cut sets with diameter bounded below, violating the point-convergence condition.  So $p\in h_v(D)$.  But this implies that (possibly after replacing $g$ with a power) $g^{-1}\big(h_v(D)\big)\subsetneq h_v(D)$.  This nesting implies that $v$ fails the DCC, which is a contradiction.
    \end{proof}

In general, the reverse inclusion to~\eqref{eq:containment} does not hold, even for $M$ equal to a Bowditch boundary $\partial (G,\mc{P})$.  For example, a proper cocompact cubulation of a relatively hyperbolic group can have vertices for which the right-hand side consists of a single parabolic point. Since the cubulation is proper, the left-hand side will be empty. On the other hand, we will see below that the reverse inclusion holds in the special case that $G$ is hyperbolic and $M = \partial G$.  

Lemma~\ref{lem:containrevised} gives the following properness criterion.
\begin{corollary}\label{cor:proper}
    Suppose, in addition to Assumption~\ref{as:main}, that $M$ is compact metric and $G$ acts on $M$ as a convergence group.  If for every vertex $v$ of the cube complex $X(\mc{D})$, we have 
    \[ \bigcap_{D\in\mc{D}} h_v(D)\cup C_D = \emptyset,\]
    then the action $G\acts X(\mc{D})$ is proper.
\end{corollary}

\subsection{The hyperbolic case}
The main goal of this subsection is to prove the following (always under Assumptions~\ref{as:main}).

\intersectionofhalfspaces*
Accordingly, for this subsection, we make the further assumption that $G$ is hyperbolic, and $M$ is equal to $\partial G$.

Our first subgoal is to show that the action of $G$ on the cube complex $X(\mc{D})$ is cocompact (Proposition~\ref{prop: hyp cocompact}). 
Closely related statements are discussed many places in the literature.  The first such result is \cite[Theorem 3.1]{Sageev97} (which relies on results in \cite{GMRS98}); see also \cite[Theorem 5]{NR03} in the setting of Coxeter groups.
We will need the following.
\begin{defobs}\label{defobs:join}
 Let $Y$ be a $\delta$--hyperbolic geodesic space, and let $\Omega$ be a subset of $\bndry Y$ containing at least two points.  The \emph{join of $\Omega$},  $\Jn(\Omega)$, is the union of all geodesic lines joining points of $\Omega$. 
  The set $\Jn(\Omega)$ is $4\delta$--quasiconvex.
\end{defobs}

\begin{lemma}\label{lem:transverse implies close join} Let $G$ be hyperbolic with generating set $S$ and $M=\bndry G$.
    Assume $D_1=(C_1,\{M_1^{+},M_1^{-}\})$ and $D_2=(C_2,\{M_2^{+},M_2^{+}\})$ be divisions of $M$ with $C_1$ and $C_2$ distinct. Suppose that $W_1=W_{D_1}$ and $W_2=W_{D_2}$ are transverse. Then there is an $R>0$ which depends only on $\mc{D}$ such that $\mc{N}_R\big(\Jn(C_1)\big)\cap\mc{N}_R\big(\Jn(C_2)\big)\neq\emptyset$ .
\end{lemma}

\begin{proof} Suppose that $W_1$ and $W_2$ are transverse. Then the four intersections $W_1^{\pm}\cap W_2^{\pm}$ are all non-empty. So there are points $a\in M_1^{+}\cap M_2^{+}$, $b\in M_1^{-}\cap M_2^{+}$, $c\in M_1^{-}\cap M_2^{-}$, and $d\in M_1^{+}\cap M_2^{-}$. Let $\gamma_{1}=[a,b]$, $\gamma_{2}=[b,c]$, $\gamma_{3}=[c,d]$, and $\gamma_{4}=[d,a]$ be bi-infinite geodesic rays joining points in $\{a,b,c,d\}$ as indicated by the pairs.  By \ref{defobs:join} both $\Jn(C_1)$ and $\Jn(C_2)$ are $4\delta$-quasiconvex. Additionally, there is a $K>8\delta$ so that $\mc{N}_i=\mc{N}_{K}\big(\Jn(C_i)\big)$ separates $\Cay(G,S)$ into at least two deep components for $i\in\{1,2\}$. The neighborhoods $\mc{N}_i$ are $2\delta$-quasiconvex.

Now, we may find vertices $g_1\in\gamma_1\cap \mc{N}_1$, and $g_3\in\gamma_3\cap \mc{N}_1$. We consider two cases: either the geodesic $[g_1,g_3]$ intersects $\mc{N}_2$, or it does not. 

First suppose that there is a vertex $p\in \mc{N}_2\cap[g_1,g_3]$. Since $\mc{N}_1$ is $2\delta$-quasiconvex, we have $d\big(p, \Jn(C_i)\big)<2\delta +K$ for $i\in\{1,2\}$.

Now, suppose $\mc{N}_2\cap[g_1,g_3]=\emptyset$. It must be that one of the bi-infinite geodesics $\gamma_1$ or $\gamma_3$ crosses $\mc{N}_2$. By re-indexing if necessary, we may assume it is $\gamma_3$. We have $\mc{N}_2$ is $2\delta$-quasiconvex, and $g_3$ is in $\mc{N}_1$. As $\gamma_3$ is geodesic, it must be that $d\big(g_3,\Jn(C_i)\big)<2\delta+2K$ for $i\in\{1,2\}$.

We may thus take $R=2\delta+2K$.
\end{proof}

\begin{definition} Let $G$ be a finitely generated  group, and fix a finite generating set (and hence a word metric) for $G$.
A subgroup $H$ has \emph{bounded packing} in $G$ if, for each constant $D$, there is a number $N = N(G,H,D)$ so that for any collection of $N$ distinct cosets
$g_1H,\ldots,g_N H$ in $G$, at least two are separated by a distance of at least $D$.
(The precise function $N$ depends on the generating set, but the existence of such a function does not.)
\end{definition}

    The following is asserted as Lemma~3.3 in~\cite{Sageev97}.  For a proof, see~\cite[Theorem 4.8]{HW09}.
\begin{theorem}\cite{Sageev97}\label{thm:bounded packing} Let $G$ be a hyperbolic group and $H$ a quasiconvex subgroup. Then $H$ has bounded packing.
\end{theorem}    

\begin{proposition}\label{prop: hyp cocompact}
    Assume $G$ is hyperbolic, $M = \partial G$, and $\mc{D}$ is a point convergent $G$--finite collection of divisions of $M$. Then the action on $X$ is cocompact.
\end{proposition}    
 
\begin{proof}
    Let $S$ be a finite generating set for $G$. We have $M=\bndry\Cay(G,S)$.
    
    We show that the cube complex is finite dimensional and that there are only finitely many orbits of cubes of any given dimension.
    
    By Proposition~\ref{prop: dimension cube complex}, $X$ is finite dimensional if there is a finite upper bound on the size of collections of pairwise transverse walls.  By Proposition~\ref{prop:qc iff null} our division stabilizers are quasiconvex. Transverse walls correspond to nearby joins of limit sets by Lemma~\ref{lem:transverse implies close join}.
    In turn, nearby joins of limit sets correspond to $R$-close cosets of division stabilizers for some $R$ depending only on the collection of divisions. 
    By Theorem~\ref{thm:bounded packing}, quasiconvex subgroups of hyperbolic groups have bounded packing. Since $\mc{D}$ is $G$-finite, only finitely many quasiconvex subgroups (and their cosets) appear.  Thus
    $X$ is finite dimensional.

    Now, every maximal cube in $X$ is in one-to-one correspondence with a family of transverse walls. Again, transverse walls correspond to $R$-close cosets of division stabilizers. Choosing $\kappa$ large enough so that the hyperplane stabilizers are all $\kappa$-quasiconvex we may apply the following fact. (The result appears in multiple places in the literature. See \cite{HW09} Section 1.2 for a discussion, or \cite[Lemma 7]{NR03} for the proof of a variation of this fact.) 
\begin{quote}
   For any $\delta,\kappa,R$, and $m$ there is an $R'$ so that:  In a $\delta$-hyperbolic space, for any collection of $m$ pairwise $R$--close $\kappa$-quasiconvex subspaces 
   $\{Q_1,\ldots, Q_m\}$ 
   there is a point $p$ such that $d(p, Q_i)\leq R'$. 
\end{quote}

    The action of $G$ on $\Cay(G,S)$ is cocompact, so there are only finitely many orbits of such points. Since $X$ is finite dimensional, we are done. 
\end{proof}

\begin{remark}\label{rem:qcstab}
    Since the division stabilizers are quasiconvex, the vertex stabilizers in the dual cube complex must also be quasiconvex, by \cite[Theorem A]{GM23}. 
\end{remark}

We will use the following construction/theorem from \cite{GM23}.
\begin{theorem}\label{thm: HGAI result}  Suppose $G$ is hyperbolic, and $G$ acts cocompactly on a cube complex $X$ with quasiconvex hyperplane stabilizers.
    Let $X_b$ be the first cubical subdivision of $X$.
    There is a locally finite connected simplicial graph $\Gamma$ acted on geometrically by $G$, and a $G$--equivariant, continuous Lipschitz map $\Psi\from \Gamma\to X_b^{(1)}$ satisfying the following:
    \begin{enumerate}
        \item\label{pairqi} Let $v$ be a vertex of $X$,
        and let $N(v)$ be the union of cubes of $X_b$ containing $b$.  Then
        both $\Psi^{-1}(v)$ and $\Psi^{-1}\big(N(v)\big)$ are connected and $G_v$--invariant.  Moreover,
        there is an equivariant quasi-isometry of pairs 
    \[ (G,G_v) \to \Big(\Gamma,\Psi^{-1}\big(N(v)\big)\Big).\]
         \item\label{intersection} Let $I$ be any non-empty intersection of hyperplanes of $X$.  Then $\Psi^{-1}(I)$ is connected.
         \item\label{lifting} Let $\alpha$ be a combinatorial path in $X^{(1)}$. Then there is a path $\tilde{\alpha}$ in $\Gamma$ such that the image of $\Psi\circ \tilde\alpha$ is equal to the image of $\alpha$.
    \end{enumerate}
\end{theorem}

\begin{proof}

    We recall in outline the main construction of \cite[Section 3]{GM23}.  We are given a group $G$ acting on a cube complex $X$, so that the non-empty chains of cubes in $X$ form a scwol $\mc{X}$.  The quotient of $\mc{X}$ by $G$ is a scwol $\mc{Y}$.  Note some of the objects of $\mc{Y}$ correspond to orbits of cubes (these are called \emph{cubical objects}), and others to orbits of chains of cubes.  
    At any rate we obtain a complex of groups $G(\mc{Y})$ whose fundamental group can be identified with $G$.  The local groups are stabilizers of certain orbit representatives of chains of cubes in $G$.  
    
    The complex of groups $G(\mathcal{Y})$ has an associated category $CG(\mathcal{Y})$ whose objects are those of $\mathcal{Y}$, and whose arrows are pairs $(g,a)$, where $a$ is an arrow of $\mathcal{Y}$, and $g$ is an element of $G_{t(a)}$.  The arrows whose second coordinate is an identity arrow are the elements of the local groups.  If the first coordinate is the identity element of $G_a$, we say the arrow is a \emph{scwol arrow}.  
    
    The graph $\Gamma$ is the $1$-skeleton of the realization of a certain subcollection of the arrows of the universal cover category $\widetilde{CG(\mc{Y})}$.  Each arrow of this universal cover has a \emph{label}, namely its image in $CG(\mc{Y})$.  

    Since the hyperplane stabilizers are assumed to be quasiconvex, \cite[Theorem A]{GM23} implies that the cell stabilizers are also quasiconvex.  In particular, they are finitely generated, so the quotient complex of groups $G(\mathcal{Y})$ has finitely generated local groups.

    For each cubical object $\tau$, we choose a finite generating set $\mathbb{A}_\tau$ for $G_\tau$.  We use the notation $\mathbb{A}$ to refer to all these choices at once.  The graph $\Gamma(\mathbb{A})$ is the realization of the set of arrows whose label is one of two types:
    \begin{enumerate}
        \item $(g,\mathbb{1}_\tau)$, where $g\in \mathbb{A}_\tau$; or
        \item $(1,a)$, where $a$ is an arrow of $\mc{Y}$ pointing from a chain $(\tau\subset \sigma)$ to either $\tau$ or $\sigma$.  The cube $\tau$ is required to be codimension exactly one in $\sigma$.
    \end{enumerate}
    A maximal subgraph of $\Gamma$ all of whose arrows are of the first type is isomorphic to the Cayley graph of some $G_\tau$ with respect to $\mathbb{A}_\tau$.  The arrows of the second type come in pairs (referred to in \cite{GM23} as \emph{pairs of opposable scwol arrows}) pointing out from a vertex of valence two.

    Next we define the map $\Psi$.  There is a functor $\Theta$ from $\widetilde{CG(\mc{Y})}$ onto $\mc{X}$ forgetting the first coordinate of each arrow $(g,a)$.  The arrows $(g,\mathbb{1}_\tau)$ are all sent to identity arrows at $\tau$; the arrows $(1,a)$ in pairs of opposable scwol arrows are sent to arrows of $\mc{X}$ whose realizations lie in the $1$--skeleton of $X_b$ (each is sent to a half-edge of that $1$--skeleton).  The map $\Psi$ is the realization of $\Theta$, restricted to the arrows of $\Gamma(\mathbb{A})$.  Since $\Theta$ is equivariant, and $\Gamma = \Gamma(\mathbb{A})$ is $G$--invariant, the map $\Psi$ is $G$--equivariant.

\begin{claim}
    For any connected subgraph $S$ in $X_b^{(1)}$, the preimage $\Psi^{-1}(S)$ is connected.
\end{claim}
 \begin{proof}
It suffices to show that if $e$ is any edge of $S$, then $\Psi^{-1}(e)$ is connected. 
    The edge $e$ connects vertices of $X_b$ corresponding to cubes $\sigma$ and $\tau$ of $X$, where $\tau$ is a codimension one face of $\sigma$.  The objects in $\Psi^{
    1}(e)$ all have label $(\overline\tau\subset\overline\sigma)$.  (Here $\overline\tau$ indicates the $G$--orbit of $\tau$, etc.)  Such an object is the source of a pair of opposable scwol arrows pointing at $\Theta^{-1}(\sigma)$ and $\Theta^{-1}(\tau)$.  In particular, $\Gamma(\mathbb{A})$ contains at least one pair of edges connecting $\Psi^{-1}(\sigma)$ to $\Psi^{-1}(\tau)$.  Since both $\Psi^{-1}(\sigma)$ and $\Psi^{-1}(\tau)$ are connected, so is $\Psi^{-1}(e)$.
 \end{proof}  
    The Claim immediately implies~\eqref{intersection} and~\eqref{lifting}.
    
    We now establish~\eqref{pairqi}.  Both $v$ and $N(v)\cap X_b^{(1)}$ are connected, so their preimages are also connected.  Since $\Psi$ is $G$-equivariant the preimages are $G_v$--invariant.  \cite[Lemma 3.12]{GM23} implies that $\Psi^{-1}\big(N(v)\big)$ is finite Hausdorff distance from $\Psi^{-1}(v)$.  The group $G$ acts geometrically on $\Gamma$; the subgroup $G_v$ acts geometrically on $\Psi^{-1}\big(N(v)\big)$.  Item~\eqref{pairqi} follows.
\end{proof}

The following lemma is true quite generally, whenever there is a cocompact action of a (not necessarily hyperbolic) group on a $\CAT(0)$ cube complex.

\begin{lemma}\label{lem: infinite stabilizer}
    Suppose the action of $G$ on $X$ is cocompact, and $H$ is a hyperplane such that $\Stab(H)=\Stab(D)$. Let $X_1$ and $X_2$ be the components of $X\setminus H$. Suppose there is some $u\in X_1^{(0)}$ with $G_u$ infinite.  Then there is an $R>0$ such that every vertex in $X_1$ is distance at most $R$ in $X_1^{(1)}$ from a vertex with infinite stabilizer.
\end{lemma}

\begin{proof}
    By cocompactness of the action of $G$ there is a compact connected subcomplex $K$ containing $u$ such that $GK=X$. Let $A=\diam(K)$ and $w\in X_1^{(0)}$.  Fix $g$ so that $w\in gK$.  If $gK$ is disjoint from $H$, then $gu$ is on the same side of $H$ as $w$ is, and $d(gu,w)\le A$.

    Suppose on the other hand that $gK$ meets $H$.  Let $\Carr(H)$ be the closed carrier of $H$.  Since $\Carr(H)$ is combinatorially convex, the closest point projection $\pi_H\from X^{(0)}\to \Carr(H)^{(0)}$ is well-defined.  
    Let $\overline{w} = \pi_H(w)$, and $\overline{u} = \pi_H(u)$.  
    Since $gK$ meets $\Carr(H)$, we have $d(w,\overline{w})\le A$.
    Let $C = d(u,\overline{u})$.  
    
    Let $G_H$ be the subgroup of $\Stab(H)$ which preserves the halfspaces $X_1, X_2$.  This subgroup is index at most $2$, and $\Stab(H)$ acts cocompactly on $H$, so $G_H$ acts cocompactly on $\Carr(H)$.  In particular there is a compact connected subcomplex $L\subset \Carr(H)\cap X_1$ so that $G_HL = \Carr(H)\cap X_1$.  Let $B$ be the diameter of $L$.  There is some $g'\in G_H$ so that $\overline{w} \in g' L$.  
    Now the distance from $w$ to $g'u$ is at most  $A + B + C$.  
\end{proof}

\begin{lemma}\label{lem:approximatesigma}
    Suppose $G$ is hyperbolic, $M=\partial M$ and $G$ acts cocompactly on $X(\mc{D})$ where $\mc{D}$ is a set of divisions satisfying Assumption~\ref{as:main}.
    
    Let $v$ be a vertex of $X(\mc{D})$.
    Let $\sigma$ be a quasi-geodesic ray in $\Gamma$ (the graph from Theorem~\ref{thm: HGAI result}), and suppose that, for some hyperplane $H$, and for some sequence of times $t_i$ going to infinity,
    $\Psi\big(\sigma(t_i)\big)$ is a vertex lying on the same side of $H$ as $v$ does.  
    
    Then $[\sigma]\in \overline{h_v(D)}$, where $D$ is the division corresponding to $H$.
\end{lemma}
\begin{proof}
    First, suppose a vertex $w$ on the same side of $H$ as $v$ is is visited infinitely often by $\Psi\big(\sigma(t_i)\big)$.  In this case $[\sigma]\in \Ld{G_w}$.  This limit set is contained in $\overline{h_w(D)}$, by Lemma~\ref{lem:containrevised}.  But since $v,w$ are on the same side of $H$, $h_v(D) = h_w(D)$, and so we are done.

     Possibly refining the sequence, we can now assume all the vertices $\{v_i = \Psi\big(\sigma(t_i)\big)\}$ are distinct.
     By Lemma~\ref{lem: infinite stabilizer} and Proposition~\ref{prop: hyp cocompact} we may choose nearby vertices $\{w_i\}$ on the same side of $H$ as $v$ so that the stabilizer of each $w_i$ is infinite.  (The fact that $w_i$ can be chosen on either side of $H$ follows from the fact that $H$ is an essential hyperplane which is skewered by the axis of a hyperbolic isometry of $X(\mc{D})$.  This follows from Proposition~\ref{prop:noglobalfp} together with the fact that there is some loxodromic for the action on $M$ with attracting fixed point in $M_D^+$ and repelling fixed point in $M_D^-$; see \cite[Theorem 2R]{Tuk94}.)

     We may also assume that the $w_i$ are chosen so that $w_i$ is a closest such vertex to $v_i$.  We claim $[\sigma]$ can be approximated by points $p_i\in \Ld{G_{w_i}}$.

    Lemma~\ref{lem: infinite stabilizer} implies that for all $i$ there is a bound on the number of vertices a path $\alpha_i$ from $v_i$ to $w_i$ in $X^{(1)}$ must pass through. By the choice of $w_i$, any vertex $u_i\neq w_i$ on $\alpha_i$ must have finite stabilizer. So, $\Psi^{-1}(u_i)$ is bounded. Thus, there exists a uniform bound $R$ such that $d\big(\Psi^{-1}(v_i),\Psi^{-1}(w_i)\big)< R$ for all $i$.

      We now choose the points $p_i$. Theorem A of \cite{GM23} and Theorem~\ref{thm: HGAI result} give that $\Psi^{-1}(w_i)$ is quasiconvex. Let $\beta_i=(\beta^-,\beta^+)$ be a bi-infinite geodesic in $\Psi^{-1}(w_i)$. Up to the action of $G_{w_i}$, we may assume that that $d\big(\beta_i(s_i),\sigma(t_i)\big)<K+R$ for some $K$ and $s_i$. Notice that since there only finitely many orbits of $w_i$ and vertices $v_i=\sigma(t_i)$, we may assume that $K$ is uniform across all $i$. This implies that one of the rays $\rho^-=\big[\sigma(0),\beta^-\big)$ or $\rho^+=\big[\sigma(0),\beta^+\big)$ pass near $\sigma(t_i)$. Let $\rho_i$ be this ray, and set $p_i=[\rho_i]$.  Then the Gromov product $(p_i|[\sigma])\rightarrow \infty$ as $i\rightarrow\infty$. Thus the distance in the visual metric $d_V(p_i,[\sigma])$ goes to zero. (We refer the reader to  \cite{BH1} Chapter III.H for details about the Gromov inner product and the visual metric.)

    Since each $\Ld{G_{w_i}}\subset \overline{h_{w_i}(D)} = \overline{h_v(D)}$,  we have that a subsequence of $p_i$ converges to in $\overline{h_v(D)}$. Therefore, $[\sigma]\in \overline{h_v(D)}$.
\end{proof}

We are ready to prove the main result of this section. 

\begin{proof}[Proof of Theorem~\ref{thm: intersection of halfspaces}] Given Lemma~\ref{lem:containrevised} we need only show the following containment to obtain the result \[\Ld{G_v}\supseteq \bigcap_{D\in\mc{D}} h_v(D)\cup C_D.\]

    To prove this we use Theorem~\ref{thm: HGAI result}. Our strategy is to show that a geodesic ray in $\Gamma$ representing a point in $\bigcap_{D\in\mc{D}} h_v(D)\cup C_D$ can be replaced with a quasigeodesic ray whose projection stays in $N(v)$.

    Let $\eta\from [0,\infty)\rightarrow \Gamma$ be a geodesic ray with $[\eta] \in \bigcap_{D\in\mc{D}} h_v(D)\cup C_D.$  Define $\mc{L}_{\eta}$ be the set of hyperplanes $H$ such that the projection $\overline{\eta}=\Psi(\eta)$ crosses $H$ and some tail of $\overline{\eta}$ is separated from $v$ by $H$.

    Let $H_0$ be some hyperplane in $\mc{L}_{\eta}$, and let $D_0$ be the corresponding division.
    Lemma~\ref{lem:approximatesigma} implies that $[\eta]\in M\minus h_v(D_0)$. Since $[\eta]\in h_v(D_0)\cup C_{D_0}$, $[\eta]$ cannot be in $M\minus\overline{h_v(D_0)}$. So, we have that $[\eta]\in C_{D_0}$ by Definition~\ref{def: division}(3). 
    This implies we can replace $\eta$ by  a quasigeodesic ray $\eta_0$ with a tail which projects to $H_0 = I_0$.  
    If the new quasigeodesic ray $\eta_0$ returns infinitely often to the cubical $1$-neighborhood of $v$ in $X_b^{(1)}$, we are done.

 We proceed inductively. Suppose there exist hyperplanes $H_0,\ldots,H_{k-1}$ meeting $N(v)$ and with intersection $I_{k-1}$ so that $\eta_{k-1}$ is a quasigeodesic ray with tail projecting to $I_{k-1}$ and $[\eta_{k-1}] = [\eta]$ in $M$. Now assume there is a hyperplane $H_k$ meeting $N(v)$ such that the tail of $\overline{\eta}_{k-1}=\Psi(\eta_{k-1})$ lies in the side of $H_k$ opposite $v$. 

    We now use Lemma~\ref{lem:approximatesigma} and the connectedness of $\Psi^{-1}(I_k)$ given by Theorem~\ref{thm: HGAI result}. The point $[\eta_{k-1}]=[\eta]$ is in $M\setminus{h_v(D_k)}$ and $\overline{h_v(D_k)}$, and thus must be in $C_{D_k}$. We may choose $\eta_k$ such that the image under $\Psi$ of some tail lies in $I_k=I_{k-1}\cap H_k$.

    As discussed in the proof of Proposition~\ref{prop: hyp cocompact} $X$ is finite dimensional. If $K$ is the dimension of $X$, then for all $n>K$ the intersection $I_n$ is empty. Thus $\overline{\eta}_k$ must be contained in $N(v)$ for some $k\leq K$.

\end{proof}

\subsection{The relatively hyperbolic case}\label{ss:rhc}
We now move to the more general setting of a relatively hyperbolic pair $(G,\mc{P})$ and $M = \partial (G,\mc{P})$.
In this subsection we prove Proposition~\ref{prop:FiniteInt} which we will use in the next section on canonical divisions, but which applies more generally. First we need a lemma. 

\begin{lemma}\label{lem: end points loxodromic}
    Let $(G,\mc{P})$ be relatively hyperbolic and $M=\bndry(G,\mc{P})$. Let $\mc{D}$ be a system of divisions which satisfies Assumptions~\ref{as:main}. Let $g$ be a loxodromic element of $G$. Then $\#\big(\Lambda(\langle g\rangle)\cap C_D\big)\in\{0,2\}$ for any cut set $C_D$.
\end{lemma}

Before we prove the lemma we remind the reader that we are in a setting where a point-convergent collection is equivalent to being a null family. (See Remark~\ref{rem: ptconv iff null} and Definition~\ref{def: null family} for more details.)

\begin{proof}
    Let $\Lambda(\langle g \rangle )=\{a=g^{+\infty},b=g^{-\infty}\}$, and let $C_D$ be the cut set corresponding to some division $D\in\mc{D}$. Suppose for a contradiction that $b\in C_D$ and $a\notin C_D$.  Fix $\epsilon >0$. Since $\langle g\rangle$ fixes $b$ and $b\in C_D$, we have that there is a $k=k(\epsilon)$ such the $\diam(g^n C_D)>\epsilon$ for all $n>k$. This violates the nullity of the collection $\{C_D\mid D\in\mathcal{D}\}$.
\end{proof}

We will need the following result of Yang, which generalizes a theorem of Susskind--Swarup in the setting of Kleinian groups \cite{SusskindSwarup}. 
\begin{theorem}\label{thm:yang12}\cite[Theorem 1.1 and Proposition 3.3]{Yang12} 
    Let $H, K$ be two relatively quasiconvex subgroups of a relatively hyperbolic
group $G$. Then
$\Ld H \cap \Ld K = \Ld{H \cap K}\sqcup E$
where the exceptional set E consists of parabolic fixed points of $\Ld{H}$ and $\Ld{K}$, whose
stabilizers in $H$ and $K$ have finite intersection.  The points of $E$ are isolated in $\Ld{H}\cap\Ld{K}$.  

\end{theorem}
Now we prove the main result of this subsection.
\begin{proposition}
\label{prop:FiniteInt}
  Let $(G,\mc{P})$ be relatively hyperbolic, and $M = \partial(G,\mc{P})$.  Let $\mc{D}$ be a system of divisions satisfying Assumption~\ref{as:main}, and suppose further that
  \begin{enumerate}
      \item For each $D\in \mc{D}$, the stabilizer of $D$ is full relatively quasiconvex.
      \item The action of $G$ on $X$ is cocompact.
  \end{enumerate}
Then for any vertex $v$ of $X^{(0)}$, either $\Ld{G_v}$ contains an off-the-wall point or $\Ld{G_v}\subseteq C_D$ for some $D\in \mc{D}$.  Moreover, $\#\{D\mid \Ld{G_v}\subseteq C_D\}<\infty$.
\end{proposition}
\begin{proof}

    We suppose that $\Ld{G_v}$ is non-empty but contains no off-the-wall point.

    Suppose $\Ld{G_v}\cap C_D\neq \emptyset$ for some $D\in\D$, and let $H_D$ be the hyperplane corresponding to the division $D$. Let $\gamma$ be a combinatorial path in $X^{(1)}$ of minimal length which starts at $v$ and crosses $H_D$. Any such path crosses $H_D$ in a unique edge. Call this edge $e$. Notice that $G_v\cap \Stab(H_D)\subset \Stab(e)$. 

    Now, let $\D_v =\set{D}{\Ld{G_v}\cap C_D\neq \emptyset}$,
    and $\mathcal{E}=\mathcal{E}(\D_v)$ be the collection of edges obtained as in the previous paragraph. Consider the sets $L_e=\Ld{G_v}\cap\Ld{G_e}$.  
    Theorem~\ref{thm:yang12} implies that $L_e$ is the union of $\Ld{G_v\cap G_e}$ with a set of isolated parabolic points $E(e)$.
\begin{claim}\label{claim:fi1}
    \begin{equation*}\Ld{G_v} \subseteq \bigcup_{e\in\mathcal{E} }L_e.\end{equation*}
    \end{claim}
    \begin{proof}[Proof of Claim~\ref{claim:fi1}]
  We are assuming that $\Ld{G_v}$ contains no off-the-wall point, so it is contained in the union of the sets $C_D$.  For a particular such set $D$, let $e = e(D)$ be the edge described above.  Then $\Ld{G_v}\cap C_D\subseteq \Ld{G_v\cap \Stab{D}}\subseteq \Ld{G_e}$, so in particular
  $\Ld{G_v}\cap C_D\subseteq L_e$.  The claim is established.
\end{proof}
  \begin{claim}\label{claim:fi2}
      Let $e\in \mc{E}$, and suppose $G_v$ is non-elementary.  Either $\Ld{G_v} = L_e$ or $L_e$ is nowhere dense in $\Ld{G_v}$.
  \end{claim}
  \begin{proof}[Proof of Claim~\ref{claim:fi2}]
      Recall from Yang's theorem \ref{thm:yang12} 
      that $L_e = \Ld{G_e\cap G_v}\sqcup E(e)$, where $E(e)$ is a set of 
      parabolic points, isolated in $L_e$, whose stabilizers in each of $G_e$ and $G_v$ are infinite.
      
      Let $H = G_e\cap G_v$.  Both $G_e$ and $G_v$ are relatively quasiconvex in $G$, so $H$ is relatively quasiconvex in $G_v$.  If $H$ is finite index in $G_v$, then $\Ld{H} = \Ld{G_v}$; moreover a parabolic point $p$ has infinite stabilizer in $G_v$ if and only if it has infinite stabilizer in $H$.  In particular this means that $E(e)$ is empty, and $L_e = \Ld{H} = \Ld{G_v}$.

      Suppose on the other hand that $H$ has infinite index in $G_v$.  For a contradiction, suppose there is some $U\subset \Ld{H}$ which is open in $\Ld{G_v}$.  This implies that $H$ is non-elementary (since $G_v$ is non-elementary).
      Because $H$ is relatively quasiconvex in $G_v$, there is some conical limit point $p$ of $H$ in $U$.  So there is a sequence of group elements $(h_i)$ and points $a\ne b$, so that $h_i p\to b$ and $h_i x\to a$ uniformly for $x\ne p$.  
      By multiplying all the $h_i$ on the left by some fixed $h$, we can assume that $a\in U$.  
      Let $K$ be the compact set $\Ld{G_v}\setminus U$, and let $y\in \Ld{G_v}\setminus\{p\}$.  For large $n$, we have $h_ny\in h_nK\subset U$.  This implies $y\in h_n U$.  Since $h_n\in H$, this means that $y\in \Ld{H}$.  In conclusion, either $H$ is finite index in $G_v$, or $\Ld{H}$ is nowhere dense in $\Ld{G_v}$.

      The set $E(e)$ is nowhere dense in $\Ld{G_v}$ since it consists of parabolic points and 
      the set of conical limit points is dense in $\Ld{G_v
    }$.
      The dichotomy follows.
  \end{proof}

    Since $\Ld{G_v}$ is compact metric, it cannot be a countable union of nowhere dense subsets, there must be some $e\in \mc{E}$ so that $\Ld{G_v} = L_e$.
    Let $\D_v'\subseteq \D_v$ be the set of those $D$ so that $\Ld{G_v}\subset C_D$ and let $\mc{E}'\subset\mc{E}$ be the corresponding collection of edges.

     Suppose now that $G_v$ is elementary.  Since  $G_v$ contains no off-the-wall points Lemma~\ref{lem: end points loxodromic} implies that $\Lambda(G_v)$ is contained in some $C_D$. Similarly, if $G_v$ is parabolic the limit set is a single point which must be in some $C_D$.

Since the stabilizer of every $D\in\mc{D}$ is assumed to be full relatively quasiconvex, Theorem 1.4 of \cite{Hruska10} implies that the  division stabilizers are undistorted. So, Theorem 1.9 of \cite{Tran19} implies that division stabilizers are strongly quasiconvex. Now, Remark 3.31 and Corollary B.(2) of \cite{GM23} imply that the fixed point sets of finite index subgroups of $G_v$ have a uniformly bounded number of cells.  Any $e\in \mc{E}'$ is in the fixed point set of some such finite index subgroup.
Thus $\mc{E}'$ is finite, and so $\D_v'$ is also finite.
\end{proof}

\section{Canonical Divisions}
\label{sec: canonical division}
The goal of this section is to prove Theorem~\ref{thm: JSJ Tree}. In the first subsection, we introduce the canonical collection of divisions associated to a family of cut sets and then make some preliminary observations including about when  $X\big(\mc{D}(\mc{C})\big)$ is a tree. 
In Subsection~\ref{subsec: Vertices of T(C)} we define the tree $\mc{T}(\mc{C})$, establish the necessary assumptions for proving Theorem~\ref{thm: JSJ Tree}, and analyze the vertex stabilizers of $\mc{T}(\mc{C})$.

Finally, in Section~\ref{subsec: connection to Bowditch tree} we explain the connection between $\mc{T}(\mc{C})$ and the cut point tree of the ``pinched'' Bowditch boundary $\bndry(G,\widehat{\mc{P}})$.  In particular we prove Theorem~\ref{thm: JSJ Tree} and Corollary~\ref{cor:homeo to jsj tree}.

\subsection{Canonical divisions and $X\big(\mc{D}(\mc{C})\big)$}\label{subsec: CD}

We recall that a \emph{cut set} in $M$ is any closed, nowhere dense subset of $M$ so that there exists some division $(C,\{A,B\})$ as in Definition~\ref{def: division}.
\begin{definition}
The \emph{valence} of a cut set $C$ is the cardinality of the set of components of $M\minus C$.
\end{definition}

\begin{definition}\label{def:D(C)}
   Suppose that $\mc{C}$ is a collection of finite valence cut sets of $M$.  The \emph{canonical collection of divisions associated to $\mc{C}$} is the set
   \[ \mc{D}(\mc{C}) = \{(C,\{A,M\minus(C\cup A)\})\mid C\in \mc{C}, A\mbox{ a component of }M\minus C\}.\]
   When $C$ has valence $\ge 3$, $A$ is a component of $M\minus C$, and $D = (C,\{A,M\minus(C\cup A)\})$, we refer to $A$ as the \emph{small} halfspace, and $M\minus(C\cup A)$ as the \emph{big} halfspace for $D$.
\end{definition}

This construction results in a small irregularity. Given a cut set in $M$ so that $M\minus C$ has $2<n<\infty$ components gives rise to $n$ different divisions. However, when $n=2$, the cut set gives only one division. This is the underlying reason for subdividing certain edges of $X\big(\mc{D}(\mc{C})\big)$ to obtain the tree $\mc{T}(\mc{C})$ (See Definition~\ref{def:T(C)} below.). When $M$ is the Bowditch boundary of relatively hyperbolic group with no peripheral splittings, we show that $\mc{T}(\mc{C})$ is equivariantly isomorphic to the cut point tree of the ``pinched'' Bowditch boundary $\widehat{M}$. The space $\widehat{M}$ is obtained by identifying cut sets in $\mc{C}$ to points.

\begin{lemma}\label{lem: satisfies assumptions}
    Suppose $M$ is path connected with no path cut point, $G\acts M$ by homeomorphisms, and $\mc{C}$ is a $G$--finite, point-convergent collection of finite valence cut sets.  Then $\mc{D}(\mc{C})$ satisfies Assumption~\ref{as:main}.
\end{lemma}
\begin{proof}
    The $G$--finiteness of $\mc{D}(\mc{C})$ follows from the $G$--finiteness of $\mc{C}$ together with the fact that each $C\in \mc{C}$ gives rise to only finitely many divisions in $\mc{D}(\mc{C})$ (since $M\minus C$ has only finitely many components).  Point-convergence of $\mc{D}(\mc{C})$ follows in the same way.
    If $D = (C,\{A,B\})$ then $\Stab(D) = \Stab(A)$ is finite index in $\Stab(C)$, again because $A$ is one of only finitely many components of $M\minus C$.  In particular, each $D\in \mc{D}(\mc{C})$ is $G$--full.  
\end{proof}

\begin{example}
 Let $\Sigma$ be a compact hyperbolic surface with connected geodesic boundary; let $Y$ be the space obtained by identifying the boundaries of three copies of $\Sigma$ with a single circle $c$.  The group $G = \pi_1(Y)$ is hyperbolic, and $\partial G$ can be identified with the visual boundary of the universal cover $\widetilde{Y}$ of $Y$.  Let $\ell$ be an elevation of $c$ to $\widetilde{Y}$ and let $C$ be the cut pair $\ell^{\pm\infty}$.
 The complement $M\setminus C$ consists of three components $M_1$, $M_2$, and $M_3$ each in distinct orbits of $\Stab(C)$. This determines three divisions $\big(C, \{M_1, M_2\cup M_3\}\big)$, $\big(C, \{M_2, M_1\cup M_3\}\big)$, and $\big(C, \{M_3, M_2\cup M_1\}\big)$. If $\mc{C}=\set{gC}{g\in G}$, then $\mc{D}=\mc{D}(\mc{C})$ is the set of $G$-translates of these three divisions.
  The cube complex $X(\mc{\mc{W}_{\mc{D}}})$ is the Bass-Serre tree for the splitting of $G$ over $\Stab(C)$.
\end{example}

\begin{definition}\label{def:mutually separate}
    Two cut sets $C,C'\in M$ are \emph{mutually non-separating} if $C'$ is contained in a connected component of $M\setminus C$, and $C$ is contained in a connected component of $M\setminus C'$. 
\end{definition}

\begin{proposition}
\label{prop: mns is tree}
   Suppose $M$ is path connected with no path cut point, $G\acts M$ by homeomorphisms, and $\mc{C}$ is a $G$--finite, point-convergent, pairwise mutually non-separating collection of finite valence cut sets.  Then $X\big(\mc{D}(\mc{C})\big)$ is a tree.
\end{proposition}

\begin{proof}
    By construction $X(\mc{D}(\mc{C}))$ is a $\CAT(0)$ cube complex, so if it is $1$--dimensional it is a tree.
    By Proposition~\ref{prop: dimension cube complex}, it is enough to show that no two walls in $\mc{W}_{\mc{D}(\mc{C})}$ are transverse.
    
      Let $W_{D_1}$ and $W_{D_2}$ be walls in $\mc{W}_{\mc{D(\mc{C})}}$ corresponding to divisions $D_1=\big(C_1, \{A_1,M\setminus (C_1\cup A_1)\}\big)$ $D_2=\big(C_2, \{A_2,M\setminus (C_2\cup A_2)\}\big)$. 
    By assumption, $C_2$ is contained in a connected component of $M\setminus C_1$, and $C_1$ is contained in a connected component of $M\setminus C_2$. Thus $C_i\subset A_j$ or $C_i\subset M\setminus(C_j\cup A_j)$ for $i,j\in\{1,2\}$ with $i\neq j$.  Let $M_i^+$ be the element of $\{A_i, M\setminus (C_i\cup A_i)\}$ which contains $C_j$, and let $M_i^-$ be the other.  To see that $W_{D_1}$ and $W_{D_2}$ are not transverse, it suffices to show that $M_1^-\cap M_2^-=\emptyset$.  Notice that $M_1^-\cap M_1^+=\emptyset$ implies that $M_1^-\cap C_2=\emptyset$. So, $C_1\cup M_1^-$ is contained in $M\setminus C_2$. Since $C_1\cup M_1^-$ is connected and $C_1\subset M_2^+$, we can conclude $C_1\cup M_1^-\subset M_2^+$.  Thus, $M_1^-\cap M_2^-=\emptyset$ as desired.  We have shown $X\big(\mc{D}(\mc{C})\big)$ is $1$--dimensional, and hence a tree.
\end{proof}

As an example application of our setup we have the following.
\begin{corollary}
   Let $(G,\mc{P})$ be a relatively hyperbolic group, and assume that $M=\bndry(G,\mc{P})$ has no cut points, and suppose  $\mc{C}$ is point-convergent, $G$--finite collection of mutually non-separating finite valence  cut sets each homeomorphic to $S^1$.
   Then $G$ splits over a quasiconvex virtual surface group.\footnote{Here "virtual surface group" includes, in the relatively hyperbolic case, the possibility that the surface is not closed, but finite type, with boundary components representing parabolic elements.}
   If the cut sets do not contain parabolic points, then $G$ splits over a virtual closed surface group.
\end{corollary}

\begin{proof}
    Proposition~\ref{prop:qc iff null} implies that the stabilizer of each $C\in\mc{C}$ is relatively quasiconvex. Then $\Stab(C)$ acts on $C$ as a geometrically finite convergence group by Definition~\ref{def: dynamic rel qc}. Then Theorem 6B of \cite{Tukia88} implies $\Stab(C)$ is virtually a surface group.  If $C$ contains no parabolic point, then $\Stab(C)$ is virtually a \emph{closed} surface group.

By Prop~\ref{prop: mns is tree} the cube complex $X = X\big(\mc{D}(\mc{C})\big)$ is a tree.  After subdividing edges if necessary we may assume that $G$ acts on $X$ without inversions. It remains to show that the action on $X$ is without fixed point. 

Consider a division $D=(C_D, \{M_D^+,M_D^-\})$ with $C_D\in\mc{C}$. We may find a loxodromic element $g\in G$ with $g^{+\infty}\in M_D^{+}$ and $g^{-\infty}\in M_D^-$ \cite[Theorem 2R]{Tuk94}. Let $U\subset M_D^{+}$ be an open neighborhood of $g^{+\infty}$, and $V\subset M_D^{-\infty}$ a neighborhood of $g^{-\infty}$. The element $g$ acts on $M$ with north-south dynamics. So there is an $n$ such that for all $k\geq n$ we have $g^k(M\setminus V)\subset U$. Thus $g^k(M_D^+)\subset M_D^+$. By Proposition~\ref{prop:noglobalfp}, $G$ acts on $X$ without a global fixed point. 
\end{proof}

\subsection{A subdivided tree $\mc{T}(\mc{C})$ and its vertex stabilizers 
when $M = \partial(G,\mc{P})$}\label{subsec: Vertices of T(C)}
\begin{definition}[The tree $\mc{T}(\mc{C})$]\label{def:T(C)}
  Suppose $G\acts M$ and $\mc{C}$ satisfy the assumptions of Proposition~\ref{prop: mns is tree}.
  Let $\mc{E}\subset\mc{C}$ be the collection of valence two cut sets and let $\mc{T}(\mc{C})$ be the tree obtained from $X\big(\mc{D}(\mc{C})\big)$ by subdividing an edge $e$ precisely when $\Stab(e)=\Stab(C)$ for some $C\in \mc{E}$.  We note that there is exactly one such edge for any element of $\mc{E}$.
\end{definition}

We make the following assumptions for the rest of Section~\ref{sec: canonical division}. 
\begin{assumption}\label{assump:RH}
    $(G,\mc{P})$ is a relatively hyperbolic group pair, $M = \partial (G,\mc{P})$ is connected with no cut points, and $\mc{C}$ is a point-convergent, $G$--finite collection of pairwise mutually non-separating finite valence cut sets.  Finally we assume that no $C\in\mc{C}$ contains an isolated parabolic point.
\end{assumption}
These assumptions imply the assumptions of Proposition~\ref{prop: mns is tree}, so the tree $\mc{T}(\mc{C})$ is defined.

The main results of this subsection are Lemma~\ref{lem: three types} and Lemma~\ref{lem: not adj}.  Lemma~\ref{lem: three types} shows that the vertices of $\mc{T}(\mc{C})$ are divided into three types; Lemma~\ref{lem: not adj} describes adjacency between these three types of vertices.

\begin{definition}\label{def: equivalence rel}
Two off-the-wall points $a$ and $b$ are \emph{equivalent} 
if there is no division $D\in\mc{D}$ so that $a\in M_D^+$ and $b\in M_D^-$.  We write this as $a\sim b$.
  
\end{definition}

\begin{definition}\label{def:principal}
    A vertex $v\in X(\mc{D})$ is \emph{principal for the equivalence class $\alpha$} if $v$ is principal (in the sense of Definition~\ref{def:principal1}) for some (hence for any) triple $(a,b,c)$ where at least two of $\{a,b,c\}$ lie in $\alpha$.
\end{definition}

\begin{definition}\label{def:semi-principal}
   Let $C$ be a cut set in $\mc{C}$ of valence at least $3$.  A vertex $v$ of $X\big(\mc{D}(\mc{C})\big)$ is \emph{semi-principal for $C$} if it is defined as follows.  Let \[D = (C_D,\{M_D^+,M_D^-\})\in \mc{D}(\mc{C})\]  If $C\ne C_D$, and $C\subset M_D^+$, we let $v(D) = W_D^+$; otherwise $v(D) = W_D^-$.  If on the other hand $C = C_D$, and $M_D^\epsilon$ is the big halfspace for $D$, we set 
   we set $v(D) = W_D^\epsilon$. In other words, we set $v(D) = W_D^+$ if $M_D^+$ is disconnected; otherwise we set $v(D) = W_D^-$.
\end{definition}

\begin{lemma}\label{lem:uniqueprincipal}
    If $v$ is principal for $\alpha$ and principal for $\beta$, then $\alpha = \beta$.
\end{lemma}
\begin{proof}
    If $x$ and $y$ are inequivalent off-the-wall points in $M$, then there is some cut set $C\in \mc{C}$ so that $x$ and $y$ are in separate components.  There is therefore some $D\in \mc{D}(\mc{C})$ so that $x$ and $y$ are in separate halfspaces for $D$.
\end{proof}
\begin{lemma}\label{lem:uniquesemiprincipal}
    If $v$ is semi-principal for $C$ and for $C'$, then $C = C'$.  Moreover $v$ is not principal.  
\end{lemma}
\begin{proof}
Suppose that $v$ is semi-principal for $C$, and suppose $C'\ne C$.  
Since $C'$ and $C$ are mutually non-separating, $C$ is contained in some component $A$ of $M\minus C'$.  In particular there is some division $D$ with $C_D = C'$ and $h_v(D) = A$.  But if $v$ were semiprincipal for $C'$, then we could not have $h_v(D)$ connected.  This contradiction implies that $v$ is not semi-principal for $C'$.  

To see that $v$ is not principal, notice that
$\bigcap \{h_v(D)\mid C_D = C\}$ is empty.
\end{proof}

\begin{lemma}\label{lem: three types}
Let $v$ be a vertex of $\mc{T}(\mc{C})$.  Then $v$ is of one of the following three mutually exclusive types.
    \begin{enumerate}
        \item\label{type:principal} principal for an equivalence class of an off-the-wall point.
        \item\label{type:semi-princ} semi-principal for some $C\in \mc{C}$ of valence at least $3$.
        \item\label{type:midpt} the midpoint of an edge corresponding to a valence two cut set.       
    \end{enumerate}   
\end{lemma}

\begin{proof} 
   Suppose $v$ is not a type \eqref{type:midpt} vertex.
   We claim first of all that $G_v$ is relatively quasiconvex and not parabolic.  Indeed, since all edge stabilizers are relatively quasiconvex, \cite[Proposition 5.2]{HH23} implies that $G_v$ is relatively quasiconvex.  If $G_v$ were parabolic, then for any edge $e$ incident to $v$, the group $G_e$ is also parabolic.  In other words, $G$ would have a peripheral splitting, and so $M$ would be either disconnected or have a cut point~\cite{Bow01_Peripheral}, contradicting Assumption~\ref{assump:RH}.
   
   By Proposition ~\ref{prop:FiniteInt} there are two cases:
   \medskip
   
        \begin{enumerate}[(i)]
        \item\label{itm:otw} There exists an off-the-wall point $p\in \Ld{G_v}$.
        \item\label{itm: no otw} $\Ld{G_v}$ is contained in a cut set $C$ of $M$.
        \end{enumerate}

\medskip
We first show that \eqref{itm:otw}
implies that $v$ is type \eqref{type:principal}. 
To begin we claim there is at least one other off-the-wall point $q$ in $\Ld{G_v}$.  Since $G_v$ is relatively quasiconvex and not parabolic, $G_v$ either contains an all-loxodromic free group of rank $2$, or $G_v$ is virtually infinite cyclic generated by a loxodromic.
In the first case, we may take any image of $p$ under an element of $G_v$ not stabilizing $p$.  Otherwise $G_v$ is virtually generated by some $g$ with attracting fixed point $p$ and repelling fixed point some $q\ne p$.  If $q$ were contained in a cut set $C$, then by Lemma~\ref{lem: end points loxodromic} we obtain a contradiction since $p$ is off-the-wall.
By Lemma~\ref{lem:containrevised}, $p\sim q$.  Finally we claim $v$ must be principal for any triple of the form $(p,q,r)$.  Let $D\in \mc{D}(\mc{C})$.  Again using Lemma~\ref{lem:containrevised}, both $p$ and $q$ must be contained in $h_v(D)$, so $v(W_D)$ is equal to the halfspace containing the triple $(p,q,r)$.

Now, we show that \eqref{itm: no otw} implies that $v$ is type \eqref{type:semi-princ}. Suppose $\Ld{G_v}\subset C$ for some $C\in\mc{C}$.  Let $D\in \mc{D}(\mc{C})$.  There are two sub-cases.  In case $C_D\neq C$, then by Lemma~\ref{lem:containrevised} $h_v(D)$ contains $C_{D_0}$.  There is only one consistent way to extend such an ultrafilter to the set $\set{W_D}{C_D = C}$, so $v$ is semi-principal for $C$. 
\end{proof}

\begin{lemma}\label{lem: not adj}
    Let $(1)$, $(2)$, and $(3)$ be the vertex types described in Lemma~\ref{lem: three types}. In $ \mc{T}(\mc{C})$: 
    \begin{enumerate}[(i)]
    \item\label{itm:sametype nonadj} vertices of the same type are not adjacent, and 
    \item\label{itm:nonprinc nonadj} vertices of type \eqref{type:semi-princ} and type \eqref{type:midpt} are never adjacent.
    \item\label{itm:correct valence} a vertex of type~\eqref{type:semi-princ} or~\eqref{type:midpt} has finite valence, equal to the valence of the corresponding cut set.
    \end{enumerate}
\end{lemma}

\begin{proof}
Notice that edges adjacent to a vertex $v$ of type~\eqref{type:principal} or type~\eqref{type:semi-princ} correspond to minimal halfspaces $h_v(D)$ (see Remark~\ref{rem:minimality}).  Edges adjacent to the type~\eqref{type:midpt} edge associated to the valence two cut set $C$ correspond to the two halfspaces $M_C^\pm$.
  
It follows from Definition~\ref{def:semi-principal} that if $v$ is semi-principal for $C$, then the divisions giving minimal halfspaces all have $C_D = C$.  In particular, the valence of $v$ (as a vertex in a graph) is the same as the valence of $C$ (as a cut set), establishing~\eqref{itm:correct valence}.  

  Next we prove~\eqref{itm:sametype nonadj}.
  Two type~\eqref{type:semi-princ} (semi-principal) vertices cannot be adjacent, since at least two halfspaces must change orientation to get between them.  

  Two type~\eqref{type:midpt} edges obviously cannot be adjacent, since each comes from subdividing a different edge of $X(\mc{C})$.  

  Finally, suppose $v$ is principal for $x$ (type~\eqref{type:principal}), and let $h_v(D)$ be the minimal halfspace for $v$ coming from an adjacent edge $e$ in $X(\mc{C})$. If $C_D$ has valence $2$, then this edge is subdivided in $\mc{T}(\mc{C})$; moving in the direction of $e$ from $v$ takes us to a type \eqref{type:midpt} vertex.  If $C_D$ has valence $k\ge 3$, then $v$ is already choosing the big halfspace for all the divisions $D'$ with $C_{D'} = C_D$, \emph{except} for $D'$, where it is choosing the small one.  Moving across the edge gets us to the semi-principal ultrafilter for $C_D$.  

  Finally we prove~\eqref{itm:nonprinc nonadj}.
  To see that vertices of types~\eqref{type:midpt} and~\eqref{type:semi-princ} cannot be adjacent, note that every edge of $X(\mc{C})$ adjacent to a type~\eqref{type:semi-princ} vertex is associated to a valence $\ge 3$ cut set.  In particular, it is not subdivided.
\end{proof}

\subsection{Connection to the cut point tree}\label{subsec: connection to Bowditch tree}
In this section we prove Theorem~\ref{thm: JSJ Tree}. The result follows from Proposition~\ref{prop:classifyv}, which gives an equivariant bijection between the vertices of $\mc{T}(\mc{C})$ and the cut point tree of $\bndry(G,\widehat{\mc{P}})$ (Section~\ref{subsec: cut point tree}).  The peripheral structure $\widehat{\mc{P}}$ is described in the following lemma.

\begin{lemma}\label{lem:extendRHstructure}
Let $\{C_1,\ldots,C_n\}$ be a set of orbit representatives in $\mc{C}$.  Let $\mc{P}'\subset \mc{P}$ consist of those parabolics whose fixed points in $M$ are not contained in any $C\in \mc{C}$.  Define \[\widehat{\mc{P}} = \mc{P}'\cup \{\Stab(C_1),\ldots\Stab(C_n)\}.\]
Then $(G,\widehat{\mc{P}})$ is relatively hyperbolic.
\end{lemma}
\begin{proof}
Note that our assumptions imply that each $C_i$ is non-degenerate and has no isolated parabolic point.  We may therefore apply
   Proposition~\ref{prop:qc iff null} to deduce that each $\Stab(C_i)$ is relatively quasiconvex with limit set equal to $C_i$.  Since the elements of $\mc{C}$ are disjoint, the collection $\widehat{\mc{P}}$ is almost malnormal (this follows from Theorem~\ref{thm:yang12}).
   We may then apply \cite[Theorem 1.1]{Yang14} to see that $(G,\widehat{\mc{P}})$ is relatively hyperbolic.
\end{proof}

\begin{lemma}\label{lem:cut sets to cut points}
    If $C\in \mc{C}$, then the image of $C$ in $\widehat{M}=\bndry(G,\widehat{\mc{P}})$ is a cut point.
\end{lemma}
\begin{proof}
     The quotient map $q\colon M\rightarrow \widehat{M}$ is given by $q(c_1)=q(c_2)$ if and only if $c_1,c_2\in C$ for some $C\in\mc{C}$. Let $C\in\mc{C}$. Let $M_1$ and $M_2$ be disjoint components of $M\setminus C$. Since $\mc{C}$ is a mutually non-separating collection of cut sets,  $q(M_1)$ and $q(M_2)$ are disjoint components of $\widehat{M}\setminus q(C)$.
\end{proof}

The following lemma follows trivially from the definition of the canonical quotient map $q\colon M\rightarrow \widehat{M}$.

\begin{lemma}\label{lem:noncut point}
    Let $q\from M\to \widehat{M}$ be the canonical quotient map.  Then $q$ restricts to a bijection between $M\minus \bigcup_{D\in\mc{D}} C_D$ and the set of non-cut points of $\widehat{M}$.  
\end{lemma}

\begin{lemma}\label{lem: image contained in a cyclic element}
    Let $\sim$ be the equivalence relation from \ref{def: equivalence rel}. The image of a $\sim$--class under the canonical quotient map $q\colon M\rightarrow \widehat{M}$ is contained in a non-trivial cyclic element of $\widehat{M}$. 
\end{lemma}
(\emph{Cyclic elements} were defined in Definition~\ref{def: cyclic element}.)
\begin{proof}
    Let $x_1$ and $x_2$ be distinct points in the same $\sim$--class in $M$. It suffices to show that $q(x_1)$ and $q(x_2)$ are not separated by a cut point in $\widehat{M}$.
    
    Suppose there is a cut point $\hat{p}$ in $\widehat{M}$ such that $q(x_1)$ and $q(x_2)$ are contained in disjoint components $\widehat{M_1}$ and $\widehat{M_2}$, respectively, of $\widehat{M}\setminus\{p\}$. Since $M$ has no cut points and $\widehat{M}=q(M)$ we have $q^{-1}(\hat{p})=C_D$ for some cut set $C_D$. However, $x_1$ and $x_2$ are in the same component $M_1$ of $M\setminus C_D$. By continuity $q(M_1)$ is connected in $\widehat{M}\setminus\{\hat{p}\}$, contradiction. 

    Notice that Lemma~\ref{lem:noncut point} implies that the cyclic element is non-trivial as $q(x_1)\neq q(x_2)$.

\end{proof}

\begin{lemma}\label{lem: bij sim to cyclic element}
    There is a one-to-one correspondence between $\sim$--classes of off-the-wall points in $M$ and non-trivial cyclic elements of $\widehat{M}$.
\end{lemma}

\begin{proof}
 Let $K$ be the collection of non-trivial cyclic elements of $\widehat{M}$ and $\Upsilon$ the collection of $\sim$--classes in $M$. By Lemma~\ref{lem: image contained in a cyclic element} the image of every $\sim$-class under $q$ is contained in some cyclic element in $K$. Let $\tau\colon \Upsilon\rightarrow K$ be the map which takes a $\sim$--class $[x]$ to the non-trivial cyclic element which contains it.

 To see surjectivity, let $\kappa\in K$ be a non-trivial cyclic element. Then by Lemma~\ref{lem:noncut point} there is an $\hat{x}\in\kappa$ such that $\hat{x}=q(x)$ for some off-the-wall point $x$. So, $\tau([x])=\kappa$ which implies $\tau$ is surjective.
 
Now, suppose $[x],[y]\in\Upsilon$ with $\tau([x])=\tau([y])$. Let $\hat{z}\in q(x)=q(z)$ be a non-cut point. Again by Lemma~\ref{lem:noncut point} $q$ restricts to a bijection on the set of off-the-wall points and non-cut points of $\widehat{M}$. Then $q^{-1}(\hat{z})\in[x]\cap[y]$. Thus $[x]=[y]$ and we have that $\tau$ is injective.
\end{proof}

\begin{lemma}\label{lemma:image closure of sim class} Let $E$ be a $\sim$--class in $M$. Then $\overline{q(E)}=\kappa$ for some cyclic element $\kappa\in K$. 
\end{lemma}

\begin{proof}
    Lemma~\ref{lem: bij sim to cyclic element} implies $q(E)$ is contained in a cyclic element $\kappa\in K$. Note that $q$ is a closed map, so $q(\overline{E})=\overline{q(E)}$. Suppose there is $\hat{x}\in \kappa\setminus \overline{q(E)}$. Then $q^{-1}(\hat{x})\notin\overline{E}$. So there is a cut set $C\in\mc{C}$ such that $\hat{x}$ is separated from every $e\in E$ by $C$. By Lemma~\ref{lem:cut sets to cut points} $q(C)$ is a cut point of $\widehat{M}$. Let $e\in E$ then $q(e)$ and $\hat{x}$ are separated by $q(C)$. Thus $\hat{x}\notin\kappa$, contradiction.
\end{proof}
\begin{lemma}\label{lem: adj}
    Let $V_1$, $V_2$, and $V_3$ be the collections of vertices of type 1, type 2, and type 3, respectively. Let $v\in V_1$ and $w\in V_2\sqcup V_3$. Suppose that $C$ is the cut corresponding to $w$ and $E$ is $\sim$--class corresponding to $v$. If $v$ and $w$ are adjacent, then $\overline{E}\cap C\neq\emptyset$.
\end{lemma}

\begin{proof}
   We will find a sequence to points of $E$ which converge to a point in $C$. Let $e$ be the edge $v$ and $w$. We have that $\Stab(e)$ is finite index in $\Stab(C)$. By  $G_v$ stabilizes $E$ by Lemma~\ref{lem:uniqueprincipal} the stabilizer $G_v$ stabilizes $E$, which implies $\Stab(e)\leq\Stab(E)$. Additionally, we know $\Stab(e)$ is finite index in $\Stab(C)$. By Proposition~\ref{prop:qc iff null} $\Stab(C)$ is relatively quasiconvex, so it contains a conical limit point for the action of $\Stab(C)$. So, there is a $\{g_n\}\subset\Stab{(e)}$ and $a,b\in C$ so that $g_n(x)\rightarrow a$ for every $x\in M\setminus {b}$. Let $x$ to be in $E$. Thus $a\in\overline{E}\cap C$.
\end{proof}

\begin{proposition}\label{prop:classifyv} Assume $G$, $\mc{C}$, $\mc{P}$ and $\widehat{\mc{P}}$ are as described in the statement of Lemma~\ref{lem:extendRHstructure}. Suppose that  $\Pi$ is the set of cut points and $K$ is the set of nontrivial cyclic elements of $\bndry(G, \widehat{\mc{P}}$). If $V$ is the vertex set of $\mc{T}(\mc{C})$, then there is an equivariant bijective map $\varphi\colon V\rightarrow \Pi\sqcup K$. Moreover, if $v$ and $w$ are adjacent in $\mc{T}(\mc{C})$, then $\varphi(v)$ and $\varphi(w)$ are adjacent in the cut point tree for $(G,\widehat{\mc{P}})$.

\end{proposition}

\begin{proof}
Let $M$, $\widehat{M}$, and $q$ be as in Lemma~\ref{lem:noncut point}.

By Lemma~\ref{lem: three types} there are three types of vertices in $V$. Let $v$ be a type (1) vertex. Then $v$ is principal for an off-the-wall point $x$. Lemma~\ref{lem: bij sim to cyclic element} implies that there is a bijection between the set of $\sim$--classes in $M$ and non-trivial cyclic elements of $\widehat{M}$. The collection of vertices $K\subset V(T)$ is the set of non-trivial cyclic elements by Theorem~\ref{thm: peripheral splittings}. Thus we have a one-to-one correspondence $\rho\colon V_1\rightarrow K$.

By Lemma~\ref{lem:cut sets to cut points} cut sets $C_D$ in $M$ are mapped to cut points in $\widehat{M}$. Moreover, the map $q$ identifies two-points if and only if they are contained in the same cut set. The elements of $\mc{C}$ are pairwise mutually non-separating, they are disjoint. Additionally, since $\mc{C}$ is a null family, the map  $q$ is upper semi-continuous by \cite{Daverman} I.2.3. It follows that there is an injective map from $V_3$, the collection of vertices of type (3), into the set $\Pi$. Lemma~\ref{lem:uniquesemiprincipal} associated a single cut set to a semi-principal vertex. So we also have that there is an injective map from $V_2$, the set vertices of type (2), into the set  $\Pi$.

Now, let $\hat{p}$ be a cut point of $M$. Then by definition of the map $q\colon M\rightarrow \hat{M}$ $q^{-1}(\hat{p})$ is contained in an element of $\mc{C}$. Thus, there is a bijection $\sigma$ between the collection of vertices of type (2) or (3) and the cut points of $\widehat{M}$.

Lastly, let $\varphi=\rho\sqcup\sigma\colon V \rightarrow \Pi\sqcup K$. 
We must show that if $v,w\in V$ are connected by an edge then so are $\varphi(v),\varphi(w)$. 

Suppose $v$ is a type (1) vertex connected by an edge $e$ to a type (2) or (3) vertex $w$.  Let $E$ be the $\sim$--class for which $v$ is principal, and let $C$ be the cut set associated to $w$.  Lemma~\ref{lem: adj} implies that $\overline{E}\cap C\neq \emptyset$.  This implies that $q(C)\cap q(\overline{E})\neq \emptyset$, which implies that $\varphi(v)$ and $\varphi(w)$ are connected by an edge by Theorem~\ref{thm: peripheral splittings}.

\end{proof}

\Bowditchtree*

\begin{proof}
    Proposition~\ref{prop:classifyv} gives an equivariant morphism of graphs from $\mc{T}(\mc{C})$ to the cut point tree, which is a bijection on the vertices.  Since both source and target are trees, this is necessarily an isomorphism.
\end{proof}

We now relate (a special case of) $\mc{T}(\mc{C})$ to the JSJ tree of cylinders, introduced by Guirardel and Levitt in~\cite{GL11}.  We will not define JSJ decompositions here, but instead refer the reader to \cite{GL17}. However we mention that the JSJ tree of cylinders over two-ended subgroups is the canonical tree which encodes all ``compatible'' splittings of $G$ over two-ended subgroups relative to $\mc{P}$.

To obtain the tree of cylinders we set $\mc{C}$ equal to a certain family of cut pairs.
\begin{definition}[inseparable and exact cut pair]
    A cut pair $\{a,b\}$ of $M$ is \emph{exact} if the number of ends of $M\setminus\{a\}$, the number of ends of $M\setminus\{b\}$, and the number of components of $M\setminus\{a,b\}$ are all equal. A cut pair $\{a,b\}$ is \emph{inseparable} if there does not exist another cut pair $\{c,d\}$ of $M$ such that $a$ and $b$ lie in different components of $M\setminus\{c,d\}$. 
\end{definition}

\begin{corollary}\label{cor:homeo to jsj tree}
    Let $(G,\mc{P})$ be a relatively hyperbolic group, and assume that $M=\bndry(G,\mc{P})$ has no cut points. 
    If $\mc{C}$ is the collection of inseparable exact cut pairs and $M$, then the cube complex dual to $\mc{D}(\mc{C})$ is equivariantly $2$--bi-Lipschitz homeomorphic to the canonical JSJ tree of cylinders over two-ended subgroups relative to $\mc{P}$.
\end{corollary}
\begin{proof}[Proof of Corollary] By Proposition 4.7 of \cite{HH23}, there are finitely many orbits of inseparable exact  cut pairs.  Note that the points in such a cut pair cannot be parabolic fixed points; in fact the stabilizer of an inseparable exact cut pair is non-parabolic two-end \cite[Proposition 4.7]{HH23}. 
By \cite{Osin06a} we may add the stabilizers of these cut pairs to the peripheral structure of $(G,\mc{P})$ to obtain a new peripheral structure $(G,\widehat{\mc{P}})$. The Bowditch boundary $\bndry(G,\widehat{\mc{P}})$ is obtained from $\bndry(G,\mc{P})$ by collapsing the exact inseparable cut pairs to points \cite{Dah03,Osin06a,Yang14}. Note that inseparability implies that images of the cut pairs in the quotient $\bndry(G,\widehat{\mc{P}})$ are cut points (\cite{H19}, Lemma 5.2). Theorem 1.1 of \cite{HH23} implies that the cut point tree of $\bndry(G,\widehat{\mc{P}})$ is equal to the JSJ tree of cylinders over two-ended subgroups relative to $\mc{P}$. The result follows from the construction of $\mc{T(\mc{C})}$ and Theorem~\ref{thm: JSJ Tree}. 

\end{proof}

\appendix
\section{Topological lemmas}

\begin{lemma}\label{lem:pathcut}
  If $X$ is a locally path connected metrizable space, every path cut point in $X$ is a cut point.
\end{lemma}
\begin{proof}
  Let $p\in X$, and let $a\ne p$.  Let $U_a$ be the set of points of $X$ which can be connected to $a$ by a path in $X\minus \{p\}$.  We claim that $U_a$ is both closed and open in $X \minus \{p\}$.

  Recall what local path connectedness means:  for any point $x\in X$ and any open neighborhood $U$ of $x$, there is a neighborhood $V$ of $x$ so that every $v\in V$ is connected to $x$ by a path in $U$.

  If $x\in U_a$, we can take $U = X \minus \{p\}$; the neighborhood $V$ as above is then also contained in $U_a$. Indeed for $y\in V$ we may concatenate the path from $a$ to $x$ with a path from $x$ to $y$ to verify $y\in U_a$.  This shows $U_a$ is open.

  To see $U_a$ is closed, suppose that $x_i\to x$ in $X\minus \{p\}$ and that $x_i\in U_a$ for all $i$.  Again, apply local path connectedness with $U = X \minus \{p\}$ to find a neighborhood $V$ of $x$.  Fix $i$ large enough so that $x_i\in V$.  This point is connected to $x$ by a path missing $p$.  Concatenating with the path which certifies $x_i\in U_a$ gives a path from $a$ to $x$.

  If $p$ is a path cut point, there is some pair of points $a$ and $b$ in $X\minus \{p\}$ so that every path from $a$ to $b$ goes through $p$.  But then $b\not\in U_a$, so $U_a$ is a nonempty proper closed and open subset of $X\minus \{p\}$.  In particular $p$ is a cut point.
\end{proof}

\begin{lemma}\label{lem: Baire}
    Suppose $M$ is a connected, nondegenerate, $T_1$, Baire space.  Let $U$ be a non-empty open subset of $M$ and let $\{K_i\}$ be a collection of closed subsets of $M$ with empty interior.  Then there are uncountably many points in $M\setminus\bigcup_i K_i$.
\end{lemma}
\begin{proof}
    Since $M$ is connected, nondegenerate, and $T_1$, no singleton is open.  Since $M$ is Baire, the open set $U$ is not contained in $\bigcup_i K_i$.  If $U\minus\bigcup_i K_i$ were countable, then we could augment the collection $\{K_i\}$ with countably many singletons to cover $U$ with countably many closed sets with empty interior.  This contradiction proves the lemma.
\end{proof}

\end{document}